\documentclass[reqno, 12pt]{amsart}
\usepackage{appendix}
\usepackage[utf8]{inputenc}
\usepackage[T1]{fontenc}
\usepackage{lmodern}
\usepackage{amsmath}
\usepackage{amsthm}
\usepackage{amssymb}
\usepackage{latexsym}
\usepackage{amsmath,amsfonts,amssymb}
\usepackage{upgreek}
\usepackage{enumerate}
\usepackage{mathrsfs}
\usepackage{mathscinet}
\usepackage{stmaryrd} 
\usepackage{graphicx}
\usepackage{xcolor}
\usepackage[all,cmtip]{xy}
\usepackage{url}
\usepackage{hyperref}
\usepackage[a4paper, hmargin=3cm, vmargin=3cm]{geometry}

\newcommand\mathens[1]{\mathbb{#1}} 

\newcommand{\ud}{\mathrm{d}}

\newcommand{\N}{\mathens{N}}
\newcommand{\Z}{\mathens{Z}}

\newcommand{\R}{\mathens{R}}
\newcommand{\C}{\mathens{C}}

\newcommand{\T}{\mathens{T}}

\newcommand\sphere[1]{\mathens{S}^{#1}}

\newcommand{\conto}{\mathrm{Cont}_{0}}

\newcommand{\tconto}{\widetilde{\mathrm{Cont}_{0}}}

\newcommand{\id}{\mathrm{id}}
\newcommand{\nlmas}\upmu

\newtheorem{thm}{Theorem}[section]
\newtheorem{lem}[thm]{Lemma}
\newtheorem{cor}[thm]{Corollary}
\newtheorem{prop}[thm]{Proposition}

\newtheorem{prop-def}[thm]{Definition-proposition}
\theoremstyle{plain}
\newtheorem*{coro*}{Corollary}

\theoremstyle{definition}

\theoremstyle{remark}

\newtheorem{exs}[thm]{Examples}
\newtheorem{rem}[thm]{Remark}

\newcommand\tpsi{\widetilde{\psi}}
\newcommand\tphi{\widetilde{\varphi}}

\newcommand\tPhi{\widetilde{\Phi}}
\newcommand\tPsi{\widetilde{\Psi}}

\newcommand\Leg{\mathcal{L}}
\newcommand\uLeg{\widetilde{\Leg}}
\newcommand\tDelta{\widetilde{\Delta}}
\newcommand\Gcont{\mathcal{G}}
\newcommand\uGcont{\widetilde{\Gcont}}
\newcommand\cleq{\preceq}

\newcommand\spec{\mathrm{Spec}}

\newcommand\tLambda{\widetilde{\Lambda}}

\DeclareFontFamily{U}{mathb}{\hyphenchar\font45}
\DeclareFontShape{U}{mathb}{m}{n}{
      <5> <6> <7> <8> <9> <10> gen * mathb
      <10.95> mathb10 <12> <14.4> <17.28> <20.74> <24.88> mathb12
}{}
\DeclareSymbolFont{mathb}{U}{mathb}{m}{n}
\DeclareMathSymbol{\cll}{3}{mathb}{"CE}

\makeatletter\let\@wraptoccontribs\wraptoccontribs\makeatother

\title[Spectral selectors and applications]{Spectral selectors on strongly orderable contact manifolds and applications}
\vspace{-2cm}
\author[P.-A. Arlove]{Pierre-Alexandre Arlove}
\address{P.-A. Arlove, Universit\'e de Strasbourg, IRMA UMR 7501, F-67000 Strasbourg, France}
\email{paarlove@unistra.fr}

\begin{document}

\begin{abstract}
We prove that the spectral selectors introduced in \cite{arlove2024contactnonsqueezingvariousclosed} for closed strongly orderable contact manifolds satisfy algebraic properties analogous to those of the spectral selectors for lens spaces constructed by Allais, Sandon, and the author using Givental’s nonlinear Maslov index. As applications, first we establish a contact big fiber theorem for closed strongly orderable contact manifolds as well as for lens spaces. Second, when the Reeb flow is periodic, we construct a stably unbounded conjugation invariant norm on the contactomorphism group universal cover. Moreover, when all its orbits have the same period, we show that the Reeb flow is a geodesic for the discriminant and oscillation norms of Colin-Sandon. Finally, we generalize a theorem of Albers-Fuchs-Merry stating that non-orderability implies the Weinstein conjecture. 
\end{abstract}

\maketitle

\section{Introduction}


Any closed cooriented contact manifold $(M,\xi)$ comes with a whole family of dynamical systems. Indeed, given a contact form $\alpha$ supporting the cooriented contact distribution $\xi$, i.e. $\ker\alpha=\xi$ and $\alpha(X)>0$ for any positive section $X$ of the oriented line bundle $TM/\xi\to M$, there exists a unique vector field $R_\alpha$, called its Reeb vector field, defined by $\alpha(R_\alpha)\equiv 1$ and $\ud \alpha(R_\alpha,\cdot)\equiv 0$. Integrating this vector field yields a one-parameter subgroup of diffeomorphisms $(\varphi_\alpha^t)_{t\in\R}$ called the Reeb flow. More precisely, for all $t\in\R$, $\varphi_\alpha^t$ is a contactomorphism, i.e. it is a diffeomorphism of $M$ that preserves the contact distribution $\xi$ and its coorientation.\\

The study of the Reeb flow dynamics is one of the driving forces in the field of contact geometry. One important question is whether any contactomorphism $\varphi$ isotopic to the identity admits a translated point $x\in M$ \cite{San11,Morse}, i.e. a point $x\in M$ such that $(\varphi^*\alpha)_x=\alpha_x$ and $\varphi(x)=\varphi_\alpha^t(x)$  for some translation $t\in\R$. While it is now known that the answer to the latter question is negative for some contact manifolds \cite{Cant2023}, in \cite{allais2023spectral} it is shown that for closed \textit{strongly orderable} contact manifolds (see Section \ref{se:spectral selector intro} and Section \ref{se:preleminaires} for a definition) the existence of translated points for contactomorphisms isotopic to the identity holds. Moreover, in the case where $(M,\xi)$ is strongly orderable, endowing the group of contactomorphisms isotopic to the identity $\conto(M,\xi)$ with the $C^1$-topology, we constructed in \cite{arlove2024contactnonsqueezingvariousclosed} a spectral selector on the universal cover $\Pi : \widetilde{\conto}(M,\xi)\to\conto(M,\xi)$, i.e. a continuous map $c:\widetilde{\conto}(M,\xi)\to\R$ such that $c(\tphi)$ is the translation of some translated point of $\Pi(\tphi)$.\\

In this paper, we first prove that the above mentioned spectral selectors possess new algebraic properties. These properties are analogous to those of the spectral selectors constructed in \cite{allais2024spectral} for the standard contact lens spaces, via Givental’s nonlinear Maslov index. 
Then, we derive from it results both on the contact topology of the underlying contact manifold and on the geometry of $\widetilde{\conto}(M,\xi)$. More precisely, in analogy with \cite{djordjević2025quantitativecontacthamiltoniandynamics,sun2025contactbigfibertheorems,uljarević2025contactquasistatesapplications}, but in the different setting of strongly orderable closed contact manifolds, we prove a \textit{contact big fiber theorem}, i.e. the impossibility to displace by contactomorphisms isotopic to the identity at least one fiber of any contact involutive map $f : M \to \mathbb{R}^N$ (see Section \ref{sec:bft intro}). Using the spectral selectors of \cite{allais2024spectral}, we obtain the analogous result in the context of lens spaces. On the other hand, when $(M,\xi)$ is strongly orderable and admits a periodic Reeb flow, these selectors allow us to define a new stably unbounded conjugation invariant norm on $\widetilde{\conto}(M,\xi)$. Moreover, if all the orbits of the periodic Reeb flow have the same minimal period, we show that the Reeb flow is a geodesic for the oscillation and discriminant norms of Colin-Sandon \cite{discriminante}. Finally, we discuss how these spectral selectors can be used to study the Weinstein conjecture. 


\subsection{Strong orderability and spectral selector}\label{se:spectral selector intro}

Before stating our main results, we recall the definition of being \textit{strongly orderable} for $(M,\xi)$. \\

To any contact form $\alpha$ supporting $\xi$, one associates  a $1$-form $\beta_\alpha$ on $G(M):=M\times M\times\R$ defined by $\beta_\alpha:=\mathrm{pr}_2^*\alpha-e^\theta\mathrm{pr}_1^*\alpha$, $\theta$ being the coordinate on $\R$ and $\mathrm{pr}_i : M \times M \times\R,\ (x_1,x_2,\theta)\mapsto x_i$ for $i\in\{1,2\}$. The distribution $\Xi_\alpha:=\ker\beta_\alpha$ turns $G(M)$ into a contact manifold and $\Delta:=\{(x,x,0)\ |\ x\in M\}$ is a Legendrian submanifold of $(G(M),\Xi_\alpha)$. The contact manifold $(M,\xi)$ is strongly orderable if, no loop of Legendrian submanifolds in $(G(M),\Xi_\alpha)$ containing $\Delta$, whose speed is everywhere transverse to the contact distribution $\Xi_\alpha$, can be contracted, through loops of Legendrians, to $\Delta$. Equivalently, any loop of Legendrians passing through $\Delta$ and moving everywhere transversely to $\Xi_\alpha$ is non-contractible in $\Leg(\Delta)$ the space of Legendrians isotopic to $\Delta$. This definition does not depend on the choice of the contact form $\alpha$ but only on the contact distribution $\xi$ (see Section \ref{se:preleminaires} for more details). \\


In the case where $(M,\xi)$ is closed and strongly orderable, for any contact form $\beta$ supporting $\Xi_\alpha$ with complete Reeb vector field, Allais and the author constructed in \cite[Theorem 1.1]{allais2023spectral}  continuous maps $\ell_-^\beta\leq\ell_+^\beta :\uLeg(\Delta)\times \uLeg(\Delta)\to\R$ (see Section \ref{se:def spec sel}), where $\uLeg(\Delta)$ stands for the universal cover of $\Leg(\Delta)$ endowed with the $C^1$-topology. Denoting by $\tDelta\in\uLeg(\Delta)$ the element that can represented by the constant Legendrian isotopy, the spectral selectors, introduced in by the author in \cite{arlove2024contactnonsqueezingvariousclosed}, and that we study in this paper are:
\begin{equation}\label{eq:def du selecteur spectral}
C_\pm^\alpha : \widetilde{\conto}(M,\xi)\to \R\quad \quad \tphi=[(\varphi_t)]\mapsto \ell_\pm^{\beta_\alpha}(\widetilde{G_\alpha}(\tphi)\cdot\tDelta,\tDelta)
\end{equation}
where $G_\alpha : \conto(M,\xi)\to \conto(G(M),\Xi_\alpha)$ is the group morphism defined by \[G_\alpha(\varphi) (x_1,x_2,\theta)\mapsto(x_1,\varphi(x_2),\theta+g(x_2))\text{ for any } \varphi\in\conto(M,\xi)\text{ with }\varphi^*\alpha=e^g\alpha\]  and $\widetilde{G_\alpha} : \widetilde{\conto}(M,\xi)\to\widetilde{\conto}(G(M),\Xi_\alpha)$ the induced group morphism, i.e. $\widetilde{G_\alpha}([(\varphi_t)])=[(G_\alpha(\varphi_t))]$.\\

In the previous definition, and throughout the paper, we identify the universal cover with the space of equivalence classes of based isotopies. Two isotopies are considered equivalent if they are homotopic relative to their endpoints.  See Section \ref{se:preleminaires} and Section \ref{se:def spec sel} for more details.\\

 For any $\tphi\in\widetilde{\conto}(M,\xi)$
 \[\spec^\alpha(\tphi):=\{t\in\R\ |\ \Pi(\tphi)(x)=\varphi_\alpha^t(x),\ (\Pi(\tphi)^*\alpha)_x=\alpha_x\}\] stands for the $\alpha$-spectrum of $\tphi$, and, for all $s\in\R$, we write $\tphi_\alpha^s:=[(\varphi_\alpha^{ts})_{t\in[0,1]}]\in\widetilde{\conto}(M,\xi)$.

\begin{thm}\label{thm:specselstrongord}
If $(M,\xi)$ is a closed strongly orderable contact manifold and $\alpha$ is a contact form supporting $\xi$, there exist continuous maps $C_-^\alpha\leq C_+^\alpha : \tconto(M,\xi)\to\R$ that satisfy the following properties for all $\tphi,\tpsi\in\widetilde{\conto}(M,\xi)$:
\begin{enumerate}[1.]
        \item $C_\pm^\alpha(\tphi)\in\spec^{\alpha}(\tphi)$ 
        \item $C_\pm^\alpha(\widetilde{\id})=0$ and $C_\pm^\alpha(\tphi_{\alpha}^s\cdot\tphi)=s+C_\pm^\alpha(\tphi)$, for all $s\in\R$
        \item if $\tphi\cleq \tpsi$ -- i.e. there exist contact isotopies $(\varphi_t),(\psi_t)$ such that $[(\varphi_t)]=\tpsi$, $[(\psi_t)]=\tpsi$ and $\alpha\left(\frac{d}{dt}\varphi_t\right)\leq\alpha\left(\frac{d}{dt}\psi_t\right)$ --  then $C_\pm^\alpha(\tphi)\leq C_\pm^\alpha(\tpsi)$
              \item $C^\alpha_+(\tphi\cdot\tpsi)\leq C_+^\alpha(\tphi)+ \ell_+^{G_\alpha(\varphi)^*\beta_\alpha}(\widetilde{G_\alpha}(\tpsi)\cdot\tDelta,\tDelta)$ and 

              \noindent
              $C^\alpha_-(\tphi\cdot\tpsi)\geq C_-^\alpha(\tphi)+ \ell_-^{G_\alpha(\varphi)^*\beta_{\alpha}}(\widetilde{G_\alpha}(\tpsi)\cdot\tDelta,\tDelta)$ where $\varphi:=\Pi(\tphi)$
        \item $C_+^\alpha(\tphi)=C_-^\alpha(\tphi)=0$ implies that $\Pi(\tphi)=\id$; 
        \end{enumerate}
     if moreover $\tphi$ can be represented by a $\alpha$-strict contact isotopy $(\varphi_t)$, i.e. $\varphi_t^*\alpha=\alpha$ for all $t$ and $[(\varphi_t)]=\tphi$, then
\begin{enumerate}[1.]\setcounter{enumi}{5}
\item $C_+^\alpha(\tphi\cdot\tpsi)\leq C_+^\alpha(\tphi)+C_+^\alpha(\tpsi)$ and $C_-^\alpha(\tphi\cdot\tpsi)\geq C_-^\alpha(\tphi)+C_-^\alpha(\tpsi)$
\item $ C_\pm^\alpha(\tphi\cdot\tpsi\cdot\tphi^{-1})=C_\pm^\alpha(\tpsi)$
\item $C_+^\alpha(\tphi)=-C_-^\alpha(\tphi^{-1})$.
    \end{enumerate}    
\end{thm}

The first four properties of the spectral selectors stated in Theorem \ref{thm:specselstrongord} are enough to prove the contact big fiber theorem for strongly orderable closed contact manifolds, stated below as Theorem \ref{thm:bft}. Moreover, Theorem \ref{thm:specselstrongord} can essentially be seen as a reformulation of Proposition 7.2 in \cite{arlove2024contactnonsqueezingvariousclosed}. The genuinely new algebraic properties of $C_\pm^\alpha$ established in this paper, which play a key role in deriving results about the geometry of $\widetilde{\conto}(M,\xi)$ (see Section \ref{se:geometry} below), are presented in the next theorem under the extra assumption that the Reeb flow is periodic. For any $T>0$, and $x\in\R$, the real numbers
\[ \left\lceil x\right\rceil_T:=T\left\lceil \frac{x}{T}\right\rceil \quad \text{ and } \quad \left\lfloor x\right\rfloor_T:=T\left\lfloor \frac{x}{T}\right\rfloor\]
are respectively the smallest multiple of $T$ greater or equal to $x$ and the biggest multiple of $T$ smaller or equal to $x$. 

\begin{thm}\label{thm:thm2}
    If $(M,\xi)$ is a closed strongly orderable contact manifold and $\alpha$ is a contact form supporting $\xi$ whose Reeb flow is $T$-periodic, for some $T>0$, the maps $C_\pm^\alpha : \tconto(M,\xi)\to\R$  satisfy moreover for any $\tphi,\tpsi\in\widetilde{\conto}(M,\xi)$:
    \begin{enumerate}[1.]
        \item $C^\alpha_+(\tphi\cdot\tpsi)\leq C_+^\alpha(\tphi)+\left\lceil C_+^\alpha(\tpsi)\right\rceil_T$ and $C^\alpha_-(\tphi\cdot\tpsi)\geq C_-^\alpha(\tphi)+\left\lfloor C_-^\alpha(\tpsi)\right\rfloor_T$
        \item $\left\lceil C_\pm^\alpha(\tpsi\cdot\tphi\cdot\tpsi^{-1})\right\rceil_T=\left\lceil C_\pm^\alpha(\tphi)\right\rceil_T$ and $\left\lfloor C_\pm^\alpha(\tpsi\cdot\tphi\cdot\tpsi^{-1})\right\rfloor_T=\left\lfloor C_\pm^\alpha(\tphi)\right\rfloor_T$
        \item $\left\lceil C_+^\alpha(\tphi)\right\rceil_T=-\left\lfloor C_-(\tphi^{-1})\right\rfloor_T$. 
    \end{enumerate}
\end{thm}

\begin{exs}\
\begin{enumerate}[1.]
    \item Closed contact manifolds admitting a Liouville filling whose symplectic homology does not vanish are strongly orderable contact manifolds \cite[Theorem 1.18]{CCR}. Thus, the unitary cotangent of a closed manifold $X$ is strongly orderable since the Liouville filling consisting of the cotangent has non vanishing symplectic homology (see Example 1.5 \cite{arlove2024contactnonsqueezingvariousclosed}). If moreover $X$ can be endowed with a Riemaniann metric with a periodic geodesic flow, then its unitary cotangent admits a periodic Reeb flow.
    \item Closed contact manifolds that are hypertight, i.e. admitting a Reeb flow without contractible periodic orbits, are strongly orderable \cite[Theorem 1.19]{CCR}. Thus, the prequantization of a closed integral symplectic manifold which is symplectically aspherical is strongly orderable and admits a periodic Reeb flow (see Example 1.5 \cite{arlove2024contactnonsqueezingvariousclosed} for more details). 
\end{enumerate}
    
\end{exs}

\subsection{Contact big fiber Theorem}\label{sec:bft intro}

The selectors $C_\pm^\alpha$ defined in formula \eqref{eq:def du selecteur spectral} above, were already used by the author in \cite{arlove2024contactnonsqueezingvariousclosed} to prove non-squeezing type theorems for closed and strongly orderable prequantizations. Here, as in \cite{djordjević2025quantitativecontacthamiltoniandynamics,uljarević2025contactquasistatesapplications}, we use our selectors to show that it is impossible to displace at least one fiber of any map $f=(f_1,\cdots,f_N):M\to\R^N$, provided that the contact flows associated to its components all commute with each other and with the Reeb flow. More precisely, there exists a vector space isomorphism between contact vector fields $\chi(M,\xi)$, i.e. vector fields whose flows yield one-parameter subgroups of contactomorphisms, and smooth functions on $M$ given by \begin{equation}
   F_\alpha: \chi(M,\xi)\to C^\infty(M,\R)\quad \quad X\mapsto \alpha(X). 
\end{equation}
Note that $F_\alpha^{-1}(1)=R_\alpha$ is the Reeb vector field of $\alpha$ and we denote by $X^\alpha_h:=F_\alpha^{-1}(h)$ for all $h\in C^\infty(M,\R)$. More explicitely, for any $h\in C^\infty(M,\R)$, $X_h^\alpha$ is the only vector field satisfying $\ud \alpha(X_h^\alpha,\cdot)=\ud h(R_\alpha)\alpha-\ud h$ and $\alpha(X_h^\alpha)=h$. 



One can pull-back the Lie bracket of $\chi(M,\xi)$ to define a Lie bracket $\{\cdot,\cdot\}_\alpha$ on $C^\infty(M,\R)$, i.e. for all $h,g\in C^\infty(M,\R)$
\[\{h,g\}_\alpha:= F_\alpha([F_\alpha^{-1}(h),F_\alpha^{-1}(g)])=\alpha([X_h^\alpha,X_g^\alpha])=\ \ud g(X^\alpha_h)-\ud h(R_\alpha) \cdot g.\] For any $N\in\N_{>0}$, a smooth function $f:M\to \R^N, x\mapsto (f_1(x),\cdots,f_N(x))$ is said to be $\alpha$-contact involutive if $\{f_i,f_j\}_\alpha=\{f_i,1\}_\alpha=0$ for all $i,j\in[1,N]\cap\N$. 

\begin{thm}\label{thm:bft}
  Let $(M,\xi)$ be a closed strongly orderable contact manifold and $N\in\N_{>0}$. If a smooth function $f:M\to\R^N$ is $\alpha$-contact involutive for some $\alpha$ supporting the contact distribution $\xi$, there is $x\in\R^N$ such that $\varphi(f^{-1}\{x\})\cap f^{-1}\{x\}\ne\emptyset$ for all $\varphi\in\conto(M,\xi)$. 
\end{thm}

 The analogous contact big fiber theorem proved in \cite[Theorem 1.4]{djordjević2025quantitativecontacthamiltoniandynamics} applies to any closed contact manifold satisfying the following assumption: the contact manifold admits a strong symplectic filling by a weakly$^+$-monotone symplectic manifold and the unit in the filling symplectic homology is not eternal. We call assumption* the latter assumption. As we discuss in Examples \ref{ex:bft}, there exist examples of closed strongly orderable contact manifolds that satisfy assumption* (such as the Ustilovsky spheres), as well as examples of closed strongly orderable contact manifolds that do not. Therefore, the contact big fiber theorem stated above in Theorem \ref{thm:bft}, that we prove in Section \ref{sec:bft}, cannot be deduced from the contact big fiber theorem in \cite{djordjević2025quantitativecontacthamiltoniandynamics}. Conversely, while assumption* implies orderabability (\cite[Theorem 1.6]{djordjević2025quantitativecontacthamiltoniandynamics}), at the time of writing, we do not know if assumption* also implies strong orderability, and hence if our contact big fiber Theorem \ref{thm:bft} implies the contact big fiber proved in \cite[Theorem 1.4]{djordjević2025quantitativecontacthamiltoniandynamics}. \\


Other contact manifolds for which we establish a contact big fiber theorem are the standard contact lens spaces. More precisely, consider the standard Euclidean sphere $\sphere{2n+1}$ of radius $1$ centered at $0$ in $\C^{n+1}\equiv \R^{2n+2}$. Fix any integer $k\geq 2$ and any $(n+1)$-tuple $\underline{w}:=(w_0,\cdots,w_{n})$ of positive integers relatively prime to $k$. The lens space $L_k^{2n+1}(\underline{w})$ is the quotient of $\sphere{2n+1}$ by the free $\Z_k$-action defined as follows: $[j]\cdot (z_0,\cdots,z_n)=(e^{\frac{2i\pi j}{k}w_0}z_0,\cdots,e^{\frac{2i\pi j}{k}w_n}z_n)$. The standard contact structure $\xi_k^{\underline{w}}$ on $L_k^{2n+1}(\underline{w})$ is given by the kernel of the unique contact form $\alpha_k^{\underline{w}}$ whose pullback to $\sphere{2n+1}$ is equal to $i^*\left(\sum\limits_{i=0}^n x_i\ud y_i-y_i\ud x_i\right)$, where $i :\sphere{2n+1}\hookrightarrow \R^{2n+2}$ denotes the inclusion and $((x_0,y_0),\cdots,(x_n,y_n))$ the coordinate functions on $\R^{2n+2}$. Note that $\varphi_{\alpha_k^{\underline{w}}}^t([(z_0,\cdots,z_n)])=[(e^{it}z_0,\cdots,e^{it}z_n)]$, for all $t\in\R$, and thus the Reeb flow is periodic and we denote by $T_{\underline{w}}>0$ its minimal period. For simplicity, in the following $\alpha$ will stand for $\alpha_k^{\underline{w}}$. 




\begin{prop}[\cite{allais2024spectral,arlove2024contactnonsqueezingvariousclosed}]\label{prop:lens space}
There exists a continuous map $c_{\alpha} : \widetilde{\conto}(L_k^{2n+1}(\underline{w}),\xi_k^{\underline{w}})\to\R$ so that for any $\tphi,\tpsi\in\widetilde{\conto}(L_k^{2n+1}(\underline{w}),\xi_k^{\underline{w}})$:
\begin{enumerate}[1.]
\item $c_{\alpha}(\tphi)\in\spec^{\alpha}(\tphi)$
\item $c_{\alpha}(\widetilde{\id})=0$ and $c_{\alpha}(\tphi_\alpha^s\cdot\tphi)=s+c_{\alpha}(\tphi)$ for all $s\in\R$
\item if $\tphi\cleq\tpsi$ -- i.e. there exist contact isotopies $(\varphi_t),(\psi_t)$ such that $[(\varphi_t)]=\tphi$, $[(\psi_t)]=\tpsi$ and $\alpha\left(\frac{d}{dt}\varphi_t\right)\leq\alpha\left(\frac{d}{dt}\psi_t\right)$ -- then $c_{\alpha}(\tphi)\leq c_{\alpha}(\tpsi)$
\item $c_{\alpha}(\tphi\cdot\tpsi)\leq c_{\alpha}(\tphi)+\left\lceil  c_{\alpha}(\tpsi)\right\rceil_{T_{\underline{w}}} $
\item if there exists a contact isotopy $(\varphi_t)_{t\in[0,1]}\subset\conto(L_k^{2n+1},\xi_k)$ such that $\varphi_t^*\alpha=\alpha$ for all $t\in[0,1]$ and  $\tphi=[(\varphi_t)]\in\widetilde{\conto}(L_k^{2n+1}(\underline{w}),\xi_k^{\underline{w}})$ then $c_{\alpha}(\tphi\cdot\tpsi)\leq c_{\alpha}(\tphi)+c_{\alpha}(\tpsi)$. 
\end{enumerate}
\end{prop}
The first four properties of the above Proposition are contained in  \cite[Theorem 1.1]{allais2024spectral} and the last one is proved in \cite[Proposition 3.3]{arlove2024contactnonsqueezingvariousclosed}. We deduce the following. 

\begin{thm}\label{thm:bft in lens space}
    Let $N\in\N_{>0}$. If a smooth function $f: L_k^{2n+1}(\underline{w})\to\R^N$ is $\alpha$-contact involutive, there is $x\in\R^N$ such that $\varphi(f^{-1}\{x\})\cap f^{-1}\{x\}\ne\emptyset$ for all $\varphi\in\conto(L_k^{2n+1}(\underline{w}),\xi_k^{\underline{w}})$ 
\end{thm}

At the time of writing, it is unclear to the author whether lens spaces are strongly orderable or satisfy assumption* in general (see Remark 1.4 in \cite{arlove2024contactnonsqueezingvariousclosed}), and thus if one can deduce Theorem \ref{thm:bft in lens space} from Theorem \ref{thm:bft} or from \cite[Theorem 1.4]{djordjević2025quantitativecontacthamiltoniandynamics}. In any case, while the existing proofs showing that a contact manifold is strongly orderable or satisfies assumption* rely on symplectic homology techniques, the spectral selectors, described in Proposition \ref{prop:lens space}, used to prove Theorem \ref{thm:bft in lens space}, are constructed via completely independent methods, namely generating functions and Givental’s nonlinear Maslov index. On the other hand,  Borman and Zapolsky proved in \cite[Corollary 1.12 and Lemma 1.22]{borman2015quasimorphisms} a result similar to Theorem \ref{thm:bft in lens space}  for the case of the real projective space ($k=2$ and $\underline{w}=(a,\cdots,a)$). Moreover, thanks to \cite[Theorem 1.17]{borman2015quasimorphisms} and the work of Granja-Karshon-Pabiniak-Sandon \cite{GKPS}, their result can be generalized to any lens space with equal weights $(k\geq 2$ and $\underline{w}=(a,\cdots,a)$). It is also interesting to note that, in contrast with these results, Marinković and Pabiniak \cite[Theorem 1.6] {marinkovic2016displaceability} proved that a generic fiber of the natural contact involutive map $f : L_k^{2n+1}(\underline{w})\to \R^{n+1}$, defined in the Example \ref{ex:bft} below, is displaceable by a contactomorphism isotopic to the identity for any $n$-tuple  $\underline{w}$.

\begin{exs}\label{ex:bft}\
\begin{enumerate}[1.]
 \item Since Ustilovsky spheres \cite{Ustilovsky1999} are contact manifolds admitting a Liouville fillings whose symplectic homology do not vanish \cite{KwonKoert2016}, they satisfy assumption* and are strongly orderable.  Thus one can apply Theorem \ref{thm:bft} to recover the contact big fiber theorem of \cite{djordjević2025quantitativecontacthamiltoniandynamics} and the non-squeezing results of \cite{Uljarevi__2023} in this case. In particular, it allows to find subsets that are displaceable by diffeomorphism isotopic to the identity but not by contactomorphism isotopic to the identity. See Section 2 in \cite{sun2025contactbigfibertheorems} for more details. 
 \item In \cite[Theorem A]{MassotNiederkrugerWendl2013}, Massot-Niederkruger-Wendl construct in all odd dimensions closed contact manifolds that are hypertight but that do not admit strong symplectic fillings. Thus, these contact manifolds are strongly orderable (Theorem 1.19 \cite{CCR}) but they do not satisfy assumption*. Therefore, these contact manifolds are examples where one can deduce a big contact fiber theorem using the above Theorem \ref{thm:bft} but not by using the techniques of \cite{djordjević2025quantitativecontacthamiltoniandynamics}. However, in this setting, the contact big fiber theorem can often be deduced by other methods. Indeed, given an hypertight contact manifold $(M,\xi)$ of dimension $2n+1$ which is toric, i.e. it admits an effective action of $\T^{n+1}$ that preserves $\xi$, and denoting by $f: M\to \R^{n+1}$ the $\alpha$-contact involutive map corresponding to the the moment map, where $\alpha$ is a $\T^{n+1}$-invariant contact form supporting $\xi$ (see \cite[Section 2]{marinkovic2016displaceability} for more details), one can use Theorem 2.3 in \cite{massot2016examples} (see also Proposition 1.5 in \cite{marinkovic2016displaceability}) to deduce that a generic fiber of $f$ cannot be displaced by any contactomorphism isotopic to the identity.  
 \item For any $j\in[0,n]\cap \N$ the contact isotopy $(\varphi_t^j)_{t\in[0,1]}\subset\conto(\sphere{2n+1},\ker\alpha_1)\subset \C^{n+1}$ defined by $\varphi_t^j(z_0,\cdots,z_n):=(z_0,\cdots, e^{it}z_j,\cdots,z_n)$ descends to a contact isotopy, that we still denote by $(\varphi_t^j)$, on $(L_k^{2n+1}(\underline{w}),\xi_k^{\underline{w}})$. A direct computation shows that $\frac{d}{dt}\varphi_t^j=X_j(\varphi_t^j)$ for all $t\in[0,1]$ where $\alpha(X_j):=f_j :L_k^{2n+1}(\underline{w})\to\R, [(z_0,\cdots,z_n)]\mapsto |z_j|^2$. Moreover $\{f_i,f_j\}_{\alpha_k}=\{1,f_i\}_{\alpha_k}$, for all $i,j\in[0,n]\cap\N$, since $\varphi_t^j\circ\varphi_s^i=\varphi_s^i\circ\varphi_t^j$ and $\varphi_t^j\circ\varphi_{\alpha_k}^s=\varphi_{\alpha_k}^s\circ\varphi_t^j$ for all $t,s\in\R$. The map $f:=(f_0,\cdots,f_n) : L_k^{2n+1}(\underline{w})\to \R^{n+1}$ is $\alpha$-contact involutive and thus has a fiber which is not displaceable by any contactomorphism isotopic to the identity by Theorem \ref{thm:bft in lens space}. Note that the fibers of $f$ are all displaceable by diffeomorphisms that are isotopic to the identity.
 
\end{enumerate}
\end{exs}

\subsection{Applications to the geometry of $\widetilde{\conto}(M,\xi)$}\label{se:geometry}
In \cite{allais2024spectral} Allais, Sandon and the author used the spectral selectors they constructed on lens spaces discussed in Proposition \ref{prop:lens space} above to deduce many results about the geometry of the contactomorphism group universal cover. We present here similar results for closed strongly orderable contact manifolds $(M,\xi)$ admitting a periodic Reeb flow.

\subsubsection{A spectral norm} 

\begin{cor}\label{cor:1}
    If $(M,\xi)$ is a closed strongly orderable contact manifold and $\alpha$ is a contact form supporting $\xi$ whose Reeb flow is $T$-periodic, for some $T>0$, the map \[\mu_{\mathrm{spec}}^\alpha : \widetilde{\conto}(M,\xi)\to\R \quad \quad \tphi\mapsto \max\left\{\left\lceil C_+^\alpha(\tphi)\right\rceil_T,-\left\lfloor C_-^\alpha(\tphi)\right\rfloor_T\right\}\]
    is a stably unbounded conjugation invariant pseudo-norm.
\end{cor}

As remarked in \cite{BIP} any conjugation invariant pseudo-norm on a group can be turned into a genuine conjugation norm by adding $1$ to the norm of all elements but the identity.\\

Note also that, when $(M,\xi)$ is strongly orderable and $\alpha$ is a contact form whose Reeb flow is $T$-periodic, the map $\mu_*^\alpha : \conto(M,\xi)\to\R$, $\varphi\mapsto \inf\{\mu_{\mathrm{spec}}^\alpha(\tphi)\ |\ \Pi(\tphi)=\varphi\}$ gives rise to a pseudo-norm on $\conto(M,\xi)$. While it is known that the group of diffeomorphisms isotopic to the identity of a closed manifold is \textit{bounded} \cite{BIP,Tsuboi}, i.e. it does not admit any unbounded conjugation invariant norm, it is still an open question whether a similar statement holds for $\conto(M,\xi)$ of any closed contact manifold $(M,\xi)$.  The author together with Allout are working to show that $\mu_*^\alpha$ is actually \textit{unbounded} when the contact manifold $M$ is a prequantization over a surface with genus. 


\subsubsection{Geodesics of the discriminant and oscillation norms}
Let us recall the definition of the oscillation and discriminant norms introduced by Colin-Sandon in \cite{discriminante}.

For any $a<b\in\R$, a contact isotopy $(\varphi_t)_{t\in[a,b]}\subset\conto(M,\xi)$ is embedded if $0\notin\spec^\alpha(\varphi_t\circ\varphi_s^{-1})$ for all $s\ne t\in[a,b]$, for some, and thus any, contact form $\alpha$ supporting $\xi$.  The discriminant norm $\nu_{\mathrm{disc}} : \widetilde{\conto}(M,\xi)\to\N$
is the word metric associated to the set $\mathcal{E}:=\{[(\varphi_t)]\in\widetilde{\conto}(M,\xi)\ |\ (\varphi_t)\text{ is  embedded and }\varphi_0=\id\}$. The discriminant length $\mathcal{L}_\mathrm{disc}$ of a contact isotopy $(\varphi_t)_{t\in[0,1]}$ is the smallest $k\in\N$ such that the isotopies $(\varphi_t)_{t\in[t_i,t_{i+1}]}$, for all $i\in[1,k-1]\cap\N$ where $0=t_1<\cdots <t_k=1$, are embedded. Moreover, by conventions the discriminant length is $0$ for the constant isotopy and $+\infty$ if there are no such $k\in\N$. Therefore $\nu_\mathrm{disc}(\tphi)=\inf\{\mathcal{L}_\mathrm{disc}(\varphi_t)\ |\ [(\varphi_t)]=\tphi\in\widetilde{\conto}(M,\xi)\}$.

\begin{cor}\label{cor:2}
    The discriminant norm is unbounded on $\widetilde{\conto}(M,\xi)$ when $(M,\xi)$ is a strongly orderable contact manifold admitting a contact form whose Reeb flow is periodic. Moreover, if all the Reeb orbits have the same minimal period then the Reeb flow is a geodesic, i.e. $\nu_{\mathrm{disc}}(\tphi_\alpha^s)=\mathcal{L}_{\mathrm{disc}}((\varphi_\alpha^{ts})_{t\in[0,1]})$ for all $s\in\R$.  
\end{cor}

A contact isotopy $(\varphi_t)\subset\conto(M,\xi)$ is monotone if $\alpha\left(\frac{d}{dt}\varphi_t\right)\ne 0$ and is positive if $\alpha\left(\frac{d}{dt}\varphi_t\right)>0$. 
One defines $\nu_{\mathrm{osc}}^+ : \widetilde{\conto}(M,\xi)\to \N$ by \[\nu_{\mathrm{osc}}^+(\tphi)=\min\left\{ k\in\N\ \left|\
        \parbox{10cm}{there exist $N\in\N$ embedded monotone isotopies $(\varphi_t^1),\ldots,(\varphi_t^N)$ starting at the identity,
            $k$ of which are positive,
    and $[(\varphi_t^1\circ\cdots\circ\varphi_t^N)]=\tphi$}\right.\right\}\] and the oscillation norm $\nu_{\mathrm{osc}} :\uGcont\to\N$ to be $\nu_{\mathrm{osc}}(\tphi):=\nu_{\mathrm{osc}}^+(\tphi)+\nu^+_{\mathrm{osc}}(\tphi^{-1})$. The positive oscillation length $\mathcal{L}^+_{\mathrm{osc}}$ of a contact isotopy $(\varphi_t)_{t\in[0,1]}$ is the smallest $k\in\N$ such that there exist $N\in\N$ and $0=t_1<\cdots<t_N=1$ satisfying the following: the contact isotopies $(\varphi_t)_{t\in(t_i,t_{i+1})}$,  for all $i\in[1,N-1]\cap \N$, are embedded and monotone, moreover exactly $k$ of them are positive. Finally, the oscillation length $\mathcal{L}_{\mathrm{osc}}$ of a contact isotopy $(\varphi_t)$ is $\mathcal{L}_{\mathrm{osc}}^+(\varphi_t)+\mathcal{L}^+_{\mathrm{osc}}(\varphi_{1-t})$. Thus, for any contact isotopy $(\varphi_t)\subset\conto(M,\xi)$ starting at the identity we have $\nu_\mathrm{osc}(\tphi)\leq\mathcal{L}_\mathrm{osc}(\varphi_t)$, where $\tphi:=[(\varphi_t)]\in\widetilde{\conto}(M,\xi)$.  


\begin{cor}\label{cor:3}
       The oscillation norm is unbounded on $\widetilde{\conto}(M,\xi)$ when $(M,\xi)$ is a strongly orderable contact maniold admitting a contact form whose Reeb flow is periodic. Moreover, if all the Reeb orbits have the same minimal period then the Reeb flow is a geodesic, i.e. $\nu_{\mathrm{osc}}(\tphi_\alpha^s)=\mathcal{L}_{\mathrm{osc}}((\varphi_\alpha^{ts})_{t\in[0,1]})$ for all $s\in\R$.
\end{cor}


\subsection{Towards applications to the Weinstein conjecture}\label{se:weinstein}

Let $(M,\xi)$ be a closed cooriented contact manifold. The Weinstein conjecture states that the Reeb flow of  any contact form supporting the contact distribution $\xi$ admits a non trivial periodic orbit.

\begin{prop}\label{prop:weinstein}
Suppose that $(M,\xi)$ is strongly orderable. Then for any contact form $\alpha$ supporting the contact distribution the map \[N_{\mathrm{spec}}^\alpha : \widetilde{\conto}(M,\xi)\to\R,\ \tphi\mapsto \max\left\{C_+^\alpha(\tphi),-C_-^\alpha(\tphi)\right\}, \] 
takes non-negative values and does not vanish on $\widetilde{\conto}(M,\xi)\setminus\pi_1(\conto(M,\xi))$. Moreover, if there exists a loop $\tphi\in\pi_1(\conto(M,\xi))$ such that $N_\mathrm{spec}^\alpha(\tphi)>0$, $(M,\xi)$ satisfies the Weinstein conjecture. 
\end{prop}

 Recall that in a closed cooriented contact manifold $(M,\xi)$ a contact isotopy $(\varphi_t)_{t\in[0,1]}\subset\conto(M,\xi)$ is said to be positive if $\alpha\left(\frac{d}{dt}\varphi_t\right)>0$ for some, and hence, any $\alpha$ supporting the contact distribution. As a direct corollary of Proposition \ref{prop:weinstein}, we recover Theorem 1.1 of \cite{AFM15} saying that non-orderability implies that the Weinstein conjecture holds.

\begin{cor}\label{cor:Weinstein}
    Let $(M,\xi)$ be a closed cooriented contact manifold. If there exists a positive loop of contactomorphisms $(\varphi_t)_{t\in[0,1]}\subset\conto(M,\xi)$, i.e. a positive contact isotopy  such that $\varphi_0=\varphi_1$, $(M,\xi)$ satisfies the Weinstein conjecture. Moreover, if this loop of contactomorphisms $(\varphi_t)$ is contractible in $\conto(M,\xi)$, the Reeb flow of any contact form $\alpha$ supporting $\xi$ has a contractible non trivial periodic orbit. 
\end{cor}

\begin{proof}[Proof of Corollary \ref{cor:Weinstein}]\

Up to multiplying the positive loop $(\varphi_t)$ by $\varphi_0^{-1}$, we may assume without loss of generality that the loop $(\varphi_t)$ starts at the identity and we denote by $\tphi:=[(\varphi_t)]\in\pi_1(\conto(M,\xi))$.

\textbf{Case 1:} Suppose that $(M,\xi)$ is hypertight. Then $(M,\xi)$ is strongly orderable by \cite[Theorem 1.19]{CCR} and thus the map $N_{\mathrm{spec}}^\alpha$ of Proposition \ref{prop:weinstein} is well defined on $\widetilde{\conto}(M,\xi)$. Positivity of the loop $(\varphi_t)$ implies that there exists $\varepsilon>0$ such that $\tphi_\alpha^{\varepsilon}\cleq \tphi$. Hence $N_{\mathrm{spec}}^\alpha(\tphi)\geq N_{\mathrm{spec}}^\alpha(\tphi_\alpha^{\varepsilon})=\varepsilon$. By Proposition \ref{prop:weinstein} we deduce that $(M,\xi)$ satisfies the Weinstein conjecture.

\textbf{Case 2:} Suppose that $(M,\xi)$ is not hypertight. Then, by definition, the Reeb flow of any contact form admits a non trivial contractible periodic orbit.\\

Moreover, suppose that the positive loop of contactomorphisms $(\varphi_t)_{t\in S^1}$ is contractible. Denote by $(\varphi_t^s)_{t\in S^1, s\in[0,1]}$ a homotopy between loops based at the identity going from $(\varphi_t)$ to the constant loop. An easy computation shows that the speed of the Legendrian loop $(\Lambda_t):=(G_\alpha(\varphi_t))$ is everywhere transverse to the contact distribution $\Xi_\alpha$. Since $(G_\alpha(\varphi_t^s)(\Delta))_{t\in S^1,s\in[0,1]}$ is an homotopy of Legendrian loops contracting  $(\Lambda_t)$ to $\Delta$ we deduce that $(M,\xi)$ is not strongly orderable. Thus, by \cite[Theorem 1.19]{CCR}, $(M,\xi)$ is not hypertight, and so the Reeb flow of any contact form admits a contractible non trivial periodic orbit. \end{proof}


Therefore, Proposition \ref{prop:weinstein} generalizes Theorem 1.1 of \cite{AFM15} since in a hypertight contact manifold the existence of a positive loop $(\varphi_t)\subset\conto(M,\xi)$ implies the existence of $\tphi\in\widetilde{\conto}(M,\xi)$ such that $N_{\mathrm{\spec}}^\alpha(\tphi)>0$ while the converse is not necessarily true. Contact manifolds where it would be natural to investigate whether one can apply the above Proposition \ref{prop:weinstein} to prove the Weinstein conjecture are Massot-Niedelkruger-Wendl hypertight contact manifolds \cite{MassotNiederkrugerWendl2013}. Indeed, these latter hypertight contact manifolds are trivial $\T^2$-bundles with a natural $S^1$-action which gives rise to an element $\tphi\in\pi_1(\conto(M,\xi))\setminus\{0\}$. Although these elements $\tphi$ are not positive (and so Theorem 1.1 of \cite{AFM15} cannot be applied in this context), the number $N_{\mathrm{spec}}^\alpha(\tphi)$ might be positive, and hence these contact manifolds would satisfy the Weinstein conjecture by virtue of Proposition \ref{prop:weinstein}. Note that, in a recent preprint \cite{oh2024contactinstantonsproofsweinsteins}, Oh has announced a proof of the Weinstein conjecture in full generality.

\subsection*{Organization of the paper}
We give the basic definitions and conventions in Section \ref{se:preleminaires}. Then in Section \ref{se:def spec sel} we define and give the important properties of the Legendrian spectral selectors $\ell_\pm^\beta$. In Section \ref{se:preuves des thm} we prove Theorem \ref{thm:specselstrongord} and Theorem \ref{thm:thm2}. In Section \ref{se:preuve des corollaires} we prove the corollaries about the geometry of $\widetilde{\conto}(M,\xi)$ stated in Section \ref{se:geometry}. In Section \ref{sec:bft} we prove the contact big fiber theorems and finally in Section \ref{se:weinstein preuve} we prove Proposition \ref{prop:weinstein}.

\subsection*{Acknowlegment} The author thanks Federica Pasquotto, Klaus Niederkrüger and Sheila Sandon for valuable discussions. 

\section{Preliminaries}\label{se:preleminaires}

Consider $(M,\xi)$ a cooriented contact manifold -- which we do not assume to be necessarily closed in this Section \ref{se:preleminaires} -- that is to say, $\xi$ is a distribution of cooriented hyperplanes on $M$ which is maximally non-integrable, i.e. there exists a $1$-form $\alpha$ on $M$ whose kernel agrees with $\xi$ and satisfies $\alpha\wedge (\ud \alpha)^n\ne \emptyset$ where $2n+1$ is the dimension of $M$, and $TM/\xi\to M$ is an oriented line bundle. A Legendrian of $(M,\xi)$ is a submanifold of $M$ of dimension $n$ which is everywhere tangent to $\xi$ and a contactomorphism of $(M,\xi)$ is a diffeomorphism of $M$ that preserves the contact distribution and its coorientation. \\

A contact isotopy $(\varphi_t)_{t\in I}$ of $(M,\xi)$, where $I\subset\R$ is an interval, is a $I$-family of contactomorphisms such that $I\times M\to M,\ (t,x)\mapsto \varphi_t(x)$ is smooth, and we denote by $\conto(M,\xi)$ the set of contactomorphisms that are contact isotopic to the identity. For a closed Legendrian $\Lambda\subset (M,\xi)$ one can define the isotopy class of $\Lambda$ to be $\Leg(\Lambda)=\{\varphi(\Lambda)\ |\ \varphi\in\conto(M,\xi)\}$. \\

On the space of contact $[0,1]$-isotopies starting at the identity $\mathcal{P}(\id)$ one can consider the equivalence relation $\sim$ of being isotopic relative to endpoints, i.e. $(\varphi_t)\sim(\psi_t)$ if there exists a smooth map $\Phi : [0,1]\times[0,1]\times M\to M$ such that $\Phi(0,t,x)=\varphi_t(x)$, $\Phi(1,t,x)=\psi_t(x)$ and $\Phi(s,0,x)=x$, $\Phi(s,1,x)=\varphi_1(x)=\psi_1(x)$ for all $s,t,x$. Timewise composition turns  $\widetilde{\conto}(M,\xi):=\mathcal{P}(\id)/\sim$ into a group, i.e. $[(\varphi_t)]\cdot[(\psi_t)]:=[(\varphi_t\circ\psi_t)]$,  the covering map $\Pi : \widetilde{\conto}(M,\xi)\to \conto(M,\xi)$, $\Pi([(\varphi_t)])=\varphi_1$ into an epimorphism. If $M$ is closed, $\widetilde{\conto}(M,\xi)$ can be canonically identified with the universal cover of $\conto(M,\xi)$ endowed with the $C^k$-topology, for any $k\in\N_{>0}\cup\{+\infty\}$. Similarly, if $\Lambda\subset (M,\xi)$ is a closed Legendrian, the space $\Leg(\Lambda)$ endowed with the $C^k$-topology is locally contractible and therefore admits a universal cover $\uLeg(\Lambda)$. As before, $\uLeg(\Lambda)$ can be identified with equivalence classes of $[0,1]$-Legendrian isotopies $(\Lambda_t)$ starting at $\Lambda$ where two Legendrian isotopies are identified if they are isotopic relative to endpoints (see \cite{allais2023spectral,arlove2025invariantdistancesspaceslegendrians} for more details). By a slight abuse of notations, we still denote by $\Pi : \uLeg(\Lambda)\to\Leg(\Lambda)$ the covering map. Moreover, there is a transitive action of $\widetilde{\conto}(M,\xi)$ on $\uLeg(\Lambda)$ given by timewise composition, i.e. $[(\varphi_t)]\cdot[(\Lambda_t)]=[(\varphi_t(\Lambda_t))]$.\\

Finally, recall that a contact isotopy $(\varphi_t)_{t\in I}\subset\conto(M,\xi)$ is non-negative if $\alpha(\frac{d}{dt}\varphi_t)\geq 0$. From this notion Eliashberg and Polterovich \cite{EP00} introduced the following natural binary relation $\cleq$ on $\widetilde{\conto}(M,\xi)$: $\tphi\cleq\tpsi$ if there exists a non-negative contact isotopy $(\varphi_t)$ such that $[(\varphi_t)]=\tphi^{-1}\cdot\psi\in\widetilde{\conto}(M,\xi)$. The binary relation $\cleq$ on $\widetilde{\conto}(M,\xi)$ is bi-invariant, i.e. $\tphi_1\cleq \tphi_2$ and $\tpsi_1\cleq \tpsi_2$ implies that $\tphi_1\cdot\tpsi_1\cleq \tphi_2\cdot \tpsi_2$, transitive and reflexive.  Similarly, one defines a binary relation $\cleq$ on $\uLeg(\Lambda)$ by: $\tLambda\cleq \tLambda'$ if there exists $\tphi\in\widetilde{\conto}(M,\xi)$ such that $\id\cleq\tphi$ and $\tphi\cdot \tLambda=\tLambda'$. The binary relation $\cleq$ on $\uLeg(\Lambda)$ is invariant (\cite[Prop 2.9]{allais2023spectral}), i.e. $\tphi_1\cleq \tphi_2$ and $\tLambda_1\cleq\tLambda_2$ implies that $\tphi_1\cdot\tLambda_1\cleq\tphi_2\cdot\tLambda_2$, transitive and reflexive. \\

When the binary relation $\cleq$ on $\widetilde{\conto}(M,\xi)$ (resp. $\uLeg(\Lambda)$) is a partial order, i.e. $\cleq$ is also antisymmetric, we say that $\widetilde{\conto}(M,\xi)$ (resp. $\uLeg(\Lambda)$) is orderable. \\

When $M$ is closed we say that $(M,\xi)$ is strongly orderable if, for some, and hence any, contact form $\alpha$ supporting $\xi$, the space $\uLeg(\Delta)$ is orderable, where $\Delta:=\{(x,x,0)\ | x\in M\}$ is the Legendrian of the contact manifold $(G(M):=M\times M\times \R,\Xi_\alpha)$ described in Section \ref{se:spectral selector intro} of the introduction. Note that $(M,\xi)$ being strongly orderable implies that $\widetilde{\conto}(M,\xi)$ and $\widetilde{\conto}(G(M),\Xi_\alpha)$ are both orderable. 


\section{Definition of the Legendrian spectral selectors}\label{se:def spec sel}

Let $(M,\xi)$ be a closed contact manifold and $\alpha$ a contact form supporting $\xi$. Consider $(G(M):=M\times M\times\R,\Xi_\alpha)$ the contact manifold described in Section \ref{se:spectral selector intro} of the introduction and $\beta$ a complete contact form supporting $\Xi_\alpha$, i.e. its Reeb vector field is complete. As for contactomorphisms, to any couple $\tLambda_1,\tLambda_2\in\uLeg(\Delta)\times\uLeg(\Delta)$ we define their $\beta$-spectrum to be:
\[ \spec^\beta(\tLambda_1,\tLambda_2):=\{t\in\R\ |\ \Pi(\tLambda_1)\cap \varphi_\beta^t(\Pi(\tLambda_2))\ne\emptyset\}.\]

In \cite{allais2023spectral}, the authors considered the maps $\ell_\pm^\beta : \uLeg(\Delta)\times \uLeg(\Delta)\to\R\cup\{\mp\infty\}$, where for all $(\tLambda_1,\tLambda_2)\in\uLeg(\Delta)\times\uLeg(\Delta)$ 
\[\ell_+^\beta(\tLambda_1,\tLambda_2):=\inf\{t\in\R\ |\ \tLambda_1\cleq\tphi_\beta^t\cdot \tLambda_2\}\quad \text{ and }\quad  \ell_-^\beta(\tLambda_1,\tLambda_2):=\sup\{t\in\R\ |\ \tphi_\beta^t\cdot\tLambda_2\cleq \tLambda_1\} \]
and showed the following.
\begin{prop}[Theorem 1.1 \cite{allais2023spectral}]\label{thm:legendrian specsel}
If $\uLeg(\Delta)$ is orderable, the maps $\ell_\pm^\beta$ take values in $\R$ are continuous and $\ell_-^\beta\leq\ell_+^\beta$. Moreover for any $\tLambda_1,\tLambda_2,\tLambda_3\in\uLeg(\Delta)$ and $\tPhi\in\widetilde{\conto}(G(M),\Xi_\alpha)$,
\begin{enumerate}[1.]
    \item $\ell^\beta_\pm(\tLambda_1,\tLambda_2)\in\spec^\beta(\tLambda_1,\tLambda_2)$
    \item $\ell^\beta_\pm(\tLambda_1,\tLambda_1)=0$ and $\ell^\beta_\pm(\tphi_\beta^s\cdot \tLambda_1,\tLambda_2)=s+\ell_\pm^\beta(\tLambda_1,\tLambda_2)$, for all $s\in\R$
    \item if $\tLambda_1\cleq\tLambda_2$, $\ell_\pm^\beta(\tLambda_1,\tLambda_3)\leq \ell_\pm^\beta(\tLambda_2,\tLambda_3)$
    \item if $\ell_+^\beta(\tLambda_1,\tLambda_2)=\ell_-^\beta(\tLambda_1,\tLambda_2)=0$,  $\Pi(\tLambda_1)=\Pi(\tLambda_2)$
    \item $\ell_+^\beta(\tLambda_1,\tLambda_3)\leq \ell_+^\beta(\tLambda_1,\tLambda_2)+\ell_+^\beta(\tLambda_2,\tLambda_3)$ and $\ell_-^\beta(\tLambda_1,\tLambda_3)\geq \ell_-^\beta(\tLambda_1,\tLambda_2)+\ell_-^\beta(\tLambda_2,\tLambda_3)$
    \item $\ell_+^\beta(\tLambda_1,\tLambda_2)=-\ell_-^\beta(\tLambda_2,\tLambda_1)$
    \item $\ell_\pm^\beta(\tPhi\cdot\tLambda_1,\tPhi\cdot\tLambda_2)=\ell_\pm^{\Pi(\tPhi)^*\beta}(\tLambda_1,\tLambda_2)$. 
\end{enumerate}
    
\end{prop}

\section{Proofs of Theorem \ref{thm:specselstrongord} and Theorem \ref{thm:thm2}}\label{se:preuves des thm}
In Section \ref{se:preuves des thm} we adopt the notation of Section \ref{se:def spec sel} and we assume that $(M,\xi)$ is a closed strongly orderable contact manifold. We fix a contact form $\alpha$ supporting $\xi$ and denote by $\beta:=\beta_\alpha$ the associated contact form on $G(M)=M\times M\times \R$.

\subsection{Proof of Theorem \ref{thm:specselstrongord}}\label{se:preuve stg}
As mentioned in the introduction, Theorem \ref{thm:specselstrongord} can essentially be seen as a reformulation of Proposition 7.2 in \cite{arlove2024contactnonsqueezingvariousclosed}. However, for completeness, in Section \ref{se:preuve stg}, we present a complete proof of Theorem \ref{thm:specselstrongord}. 

\begin{proof}[Proof of Theorem \ref{thm:specselstrongord}]\
Since the map $\widetilde{\conto}(M,\xi)\to \uLeg(\Delta)$, $\tphi\mapsto \widetilde{G_\alpha}(\tphi)\cdot \widetilde{\Delta}$ is continuous, we deduce the continuity of $C_\pm^\alpha$. Moreover, since $\ell_-^\beta\leq\ell_+^\beta$ we deduce that $C_-^\alpha\leq C_+^\alpha$. Let us now prove the seven properties listed in Theorem \ref{thm:specselstrongord}.
\begin{enumerate}[1.]
\item By a direct computation we get that $\varphi_\beta^t(x_1,x_2,\theta)=(x_1,\varphi_\alpha^t(x_2),\theta)$ for all $t\in\R$ and all $(x_1,x_2,\theta)\in G(M)$. Thus $\spec^\alpha(\tphi)=\spec^\beta(\widetilde{G_\alpha}(\tphi)\cdot \tLambda,\tLambda)$, and therefore $C_\pm^\alpha(\tphi)\in\spec^\alpha(\tphi)$ by the first property of $\ell_\pm^\beta$ stated in Theorem \ref{thm:legendrian specsel}.
\item Since $\ell_\pm^\beta(\tDelta,\tDelta)=0$, we deduce that $C_\pm^\alpha(\widetilde{\id})=0$. Moreover, by what we have said in the previous point, $\widetilde{G_\alpha}(\tphi_\alpha^t)=\tphi_\beta^t$ for all $t\in\R$. Therefore for any $\tphi\in\widetilde{\conto}(M,\xi)$, denoting by $\tPhi:=\widetilde{G_\alpha}(\tphi)$ we get the desired equality
\[C_\pm^\alpha(\tphi_\alpha^t\cdot\tphi)=\ell_\pm^\beta(\tphi_\beta^t\cdot \tPhi\cdot\tDelta,\tDelta)=t+\ell_\pm^\beta(\tPhi\cdot\tDelta,\tDelta)=t+C_\pm^\alpha(\tphi),\quad \text{for all }t\in\R.\]
\item Let $(\varphi_t)\subset\conto(M,\xi)$ be a contact isotopy and denote by $(\Phi_t):=(G_\alpha(\varphi_t))$ the corresponding contact isotopy in $\conto(G(M),\Xi_\alpha)$. By a direct computation   we get $\alpha\left(\frac{d}{dt}\varphi_t(x_2)\right)=\beta\left(\frac{d}{dt}\Phi_t(x_1,x_2,\theta)\right)$ for all $(x_1,x_2,\theta)\in M\times M\times \R$. In particular, for any $\tphi_1,\tphi_2\in\widetilde{\conto}(M,\xi)$, if $\tphi_1\cleq\tphi_2$ then $\tPhi_1\cleq\tPhi_2$ where $\tPhi_1=\widetilde{G_\alpha}(\tphi_1),\tPhi_2=\widetilde{G_\alpha}(\tphi_2)\in\widetilde{\conto}(G(M),\Xi_\alpha)$. Thus by invariance of the binary relation on $\uLeg(\Delta)$ (\cite[Proposition 2.9]{allais2023spectral}) and the third property of $\ell^\beta_\pm$ stated in Theorem \ref{thm:legendrian specsel} we deduce the desired inequality $C_\pm^\alpha(\tphi_1)\leq C_\pm^\alpha(\tphi_2)$. 
\item Let $\tphi,\tpsi\in\widetilde{\conto}(M,\xi)$ and denote by $\tPhi:=\widetilde{G_\alpha}(\tphi)$, $\tPsi:=\widetilde{G_\alpha}(\tpsi)$. Thus 
\[\begin{aligned}C_+^\alpha(\tphi\cdot\tpsi)=\ell_+^\beta(\tPhi\cdot\tPsi\cdot\tDelta,\tDelta)&\leq \ell_+^\beta(\tPhi\cdot\tPsi\cdot \tDelta,\tPhi\cdot\tDelta)+\ell_+^\beta(\tPhi\cdot\tDelta,\tDelta)\\
&=\ell_+^{\Pi(\tPhi)^*\beta}(\tPsi\cdot\tDelta,\tDelta)+C_+^\alpha(\tphi),
\end{aligned}\]
where the inequality follows from the fifth property of $\ell_+^\beta$ and the last equality from the seventh property of $\ell_+^\beta$. Since $\Pi(\tPhi)=\Pi(\widetilde{G}_\alpha(\tphi))=G_\alpha(\Pi(\tphi))$ this proves the triangle inequality satisfied by $C_+^\alpha$. A similar argument allows to prove the inequality satisfied by $C_-^\alpha$. 
\item   Let $\tphi\in\widetilde{\conto}(M,\xi)$ so that $C_\pm^\alpha(\tphi)=0$. It implies by the fourth property of $\ell_\pm^\beta$ that $\Pi(\tPhi\cdot\tDelta)=\Pi(\tDelta)=\Delta$, where $\tPhi=\widetilde{G_\alpha}(\tphi)$. Since $\Pi(\tPhi\cdot\tDelta)=\{(x,\Pi(\tphi)(x),g(x)\ |\ x\in M\}=\{(x,x,0)\ | x\in M\}$, where $g$ is the $\alpha$-conformal factor of $\Pi(\tphi)$, i.e. $\Pi(\tphi)^*\alpha=e^g\alpha$, we deduce that $\Pi(\tphi)=\id$.

\item We prove the inequality only for $C_+^\alpha$ since the same argument applies to $C_-^\alpha$. An easy computation shows that if $g:M\to\R$ is the $\alpha$-conformal factor of $\varphi\in\conto(M,\xi)$ then $\overline{g} : G(M)\to \R,\ (x_1,x_2,\theta)\mapsto g(x_2)$ is the $\beta$-conformal factor of $G_\alpha(\varphi)\in\conto(G(M),\Xi_\alpha)$. In particular if $\varphi^*\alpha =\alpha$, then $G_\alpha(\varphi)^*\beta=\beta$. Thus, if $\tphi\in\widetilde{\conto}(M,\xi)$ is such that $\Pi(\tphi)^*\alpha=\alpha$, by the triangle inequality proved at the point 4, we get the desired inequality
\[C_+^\alpha(\tphi\cdot\tpsi)\leq C_+^\alpha(\tphi)+\ell_+^\beta(\widetilde{G_\alpha}(\tpsi)\cdot \tDelta,\tDelta)=C_+^\alpha(\tphi)+C_+^\alpha(\tpsi)\]
for any $\tpsi\in\widetilde{\conto}(M,\xi)$.

\item Consider $\tphi\in\widetilde{\conto}(M,\xi)$ so that it can be represented by a $\alpha$-strict contact isotopy $(\varphi_t)$. Then $\spec^\alpha(\tphi_s\cdot\tpsi\cdot\tphi_s^{-1})=\spec^\alpha(\tpsi)$ for all $s\in[0,1]$, where $\tphi_s:=[(\varphi_{ts})_{t\in[0,1]}]$. By continuity of $C_\pm^\alpha$ and the nowhere density of $\spec^\alpha(\tpsi)$ (see \cite[Lemma 2.11]{albers}) we deduce that $s\mapsto C_\pm^\alpha(\tphi_s\cdot\tpsi\cdot\tphi_s^{-1})$ is constant and therefore $C_\pm^\alpha(\tphi\cdot\tpsi\cdot\tphi^{-1})=C_\pm^\alpha(\tpsi)$.
\item  If $\tphi\in\widetilde{\conto}(M,\xi)$ is such that $\Pi(\tphi)^*\alpha=\alpha$, denoting by $\tPhi:=\widetilde{G_\alpha}(\tphi)\in\widetilde{\conto}(G(M),\Xi_\alpha)$, we get that $\Pi(\tPhi)^*\beta=\beta$ by what we have said in the sixth point of this proof. Therefore, $\ell_+^\beta(\tPhi\cdot \tDelta,\tDelta)=\ell_+^{\Pi(\tPhi)^*\beta}(\tPhi^{-1}\cdot\tPhi\cdot \tDelta,\tPhi^{-1}\cdot\tDelta)=\ell_+^\beta(\tDelta,\tPhi^{-1}\cdot\tDelta)$ and hence
\[C_+^\alpha(\tphi)=\ell_+^\beta(\tPhi\cdot \tDelta,\tDelta)=\ell_+^\beta(\tDelta,\tPhi^{-1}\cdot\tDelta)=-\ell_-^\beta(\tPhi^{-1}\cdot\tDelta,\tDelta)=-C_-^\alpha(\tphi^{-1}),\]
where the third equality is a consequence of the sixth property of $\ell_\pm^\beta$. 
\end{enumerate}
\end{proof}

\subsection{Proof of Theorem \ref{thm:thm2}}\label{se:preuve thm2}

In this Subsection \ref{se:preuve thm2} we assume morever that the Reeb flow of $\alpha$ on $(M,\xi)$ is $T$-periodic for some $T>0$. Therefore the Reeb flow of $\beta=\beta_\alpha$ is also $T$-periodic on $G(M)=M\times M\times \R$, since $\varphi_\beta^t(x_1,x_2,\theta)=(x_1,\varphi_\alpha^t(x_2),\theta)$ for all $t\in\R$ and all $(x_1,x_2,\theta)\in G(M)$. In particular, $\tphi_\beta^{NT}\in\pi_1(\conto(G(M),\Xi_\alpha))$ lies in the center of $\widetilde{\conto}(G(M),\Xi_\alpha)$ for all $N\in\Z$. We will prove the that $C_\pm^\alpha$ have the three properties stated in Theorem \ref{thm:thm2} in the three next lemmas. 

\begin{lem}\label{lem:poincare}
    $\left\lceil C_+^\alpha(\tphi)\right\rceil_T=-\left\lfloor C_-^\alpha(\tphi^{-1})\right\rfloor_T$ for all $\tphi\in\widetilde{\conto}(M,\xi)$.
\end{lem}

\begin{proof}
    Let $N\in\Z$ such that $NT:=\left\lceil C_+^\alpha(\tphi)\right\rceil_T$.

\textbf{First case:} Suppose that $\spec^\alpha(\tphi)\cap T\Z=\emptyset$. Thus $C_+^\alpha(\tphi)<NT$ and denoting by $\tPhi=\widetilde{G_\alpha}(\tphi)\in\widetilde{\conto}(G(M),\Xi_\alpha)$  we have the following series of implications 
\[\tPhi\cdot\tDelta\cleq\tphi_{\beta}^{NT}\cdot \tDelta\quad \Rightarrow\quad   \tDelta\cleq \tPhi^{-1}\cdot\tphi_\beta^{NT}\cdot\tDelta = \tphi_\beta^{NT}\cdot \tPhi^{-1}\cdot\tDelta\quad \Rightarrow\quad  \tphi_\beta^{-NT}\cdot \tDelta\cleq \tPhi^{-1}\cdot\tDelta.\]
Therefore $\left\lfloor C_-^\alpha(\tphi^{-1})\right\rfloor_T \geq -NT$. If we suppose by contradiction that $\left\lfloor C_-^\alpha(\tphi^{-1})\right\rfloor_T > -NT$ then $C_-^\alpha(\tphi^{-1})>(1-N)T$ since $\spec^\alpha(\tphi^{-1})\cap T\Z=\emptyset$. This would then imply the following series of implications \[\tphi_\beta^{(1-N)T}\cdot \tDelta \cleq \tPhi^{-1}\cdot \tDelta\ \Rightarrow \ \tPhi\cdot \tphi_\beta^{(1-N)T}\cdot \tDelta \cleq \tDelta\ \Rightarrow\ \tphi_\beta^{(1-N)T}\cdot \tPhi\cdot \tDelta\cleq \tDelta\ \Rightarrow\ \tPhi\cdot \tDelta\cleq \tphi_\beta^{(N-1)T}\cdot \tDelta\]
which is a contradiction since $C_+^\alpha(\tphi)>N-1$. Hence, we deduce, in this case, that $\left\lfloor C_-^\alpha(\tphi^{-1})\right\rfloor_T=-NT=-\left\lceil C_+^\alpha(\tphi)\right\rceil_T$, which proves the Lemma.

\textbf{Second case:} Suppose now $\spec^\alpha(\tphi)\cap T\Z\ne\emptyset$. Since $\spec^\alpha(\tphi)$ is nowhere dense by Lemma 2.11 in \cite{albers}, there exists a decreasing sequence of positive real numbers $(\varepsilon_p)_{p\in\N}$ going to $0$ so that, denoting by $\tphi_p:=\tphi_\alpha^{-\varepsilon_p} \cdot\tphi$ for all $p\in\N$, $\spec^\alpha(\tphi_p)\cap T\Z=\emptyset$.
Thus by the first part of the proof $\left\lceil C_+^\alpha(\tphi_p)\right\rceil_T=-\left\lfloor C_-^\alpha(\tphi_p^{-1})\right\rfloor_T$, for all $p\in\N$. By continuity and monotonicity of the selectors, for $p$ big enough $\left\lceil C_+^\alpha(\tphi_p)\right\rceil_T=\left\lceil C_+^\alpha(\tphi)\right\rceil_T$ and $\left\lfloor C_-^\alpha(\tphi_p^{-1})\right\rfloor_T=\left\lfloor C_-^\alpha(\tphi^{-1})\right\rfloor_T$ which concludes the proof. 
\end{proof}

\begin{lem}\label{lem:inegalite triangulaire}
$C^\alpha_+(\tpsi\cdot\tphi)\leq C_+^\alpha(\tpsi)+\left\lceil C_+^\alpha(\tphi)\right\rceil_T$ and $C^\alpha_-(\tpsi\cdot\tphi)\geq C_-^\alpha(\tpsi)+\left\lfloor C_-^\alpha(\tphi)\right\rfloor_T$ for all $\tphi,\tpsi\in\widetilde{\conto}(M,\xi)$
\end{lem}

\begin{proof}
    We prove the inequality only for $C_+^\alpha$ since the proof for $C_-^\alpha$ is similar. Let $N\in\Z$ such that $NT:=\left\lceil C_+^\alpha(\tphi)\right\rceil_T$, $s:=C_+^\alpha(\tpsi)$ and $\varepsilon>0$.

    \textbf{First case:} Suppose $\spec^\alpha(\tphi)\cap T\Z =\emptyset$. Thus $C_+^\alpha(\tphi)<NT$, and denoting by $\tPhi:=\widetilde{G_\alpha}(\tphi),\tPsi:=\widetilde{G_\alpha}(\tpsi)\in\widetilde{\conto}(G(M),\Xi_\alpha)$ we have: $\tPhi\cdot\tDelta\cleq\tphi_\beta^{NT}\cdot\tDelta$. By multiplying both sides by $\tPsi:=\widetilde{G_\alpha}(\tpsi)\in\widetilde{\conto}(N,\Xi_\alpha)$ we get \[\tPsi\cdot\tPhi\cdot\tDelta\cleq \tPsi\cdot\tphi_\beta^{NT}\cdot\tDelta=\tphi_\beta^{NT}\cdot\tPsi\cdot\tDelta\cleq \tphi_\beta^{NT}\cdot\tphi_\beta^{s+\varepsilon}\cdot\tDelta=\tphi_\beta^{NT+s+\varepsilon}\cdot\tDelta.\] Since $\varepsilon$ can be choosen arbitrarily small, we get the desired inequality $C^\alpha_+(\tpsi\cdot\tphi)\leq C_+^\alpha(\tpsi)+\left\lceil C_+^\alpha(\tphi)\right\rceil_T$. 
    
    \textbf{Second case:} Suppose $\spec^\alpha(\tphi)\cap T\Z\ne\emptyset$. As in the proof of Lemma \ref{lem:poincare} since $\spec^\alpha(\tphi)$ is nowhere dense, one can find a decreasing sequence of positive real numbers $(\varepsilon_p)_{p\in\N}$ going to $0$ so that denoting by $\tphi_p:=\tphi_\alpha^{-\varepsilon_p}\cdot\tphi$ for all $p\in\N$, we have $\spec^\alpha(\tphi_p)\cap T\Z=\emptyset$. Hence, using the first part of the proof, $C_+^\alpha(\tpsi\cdot\tphi_p)\leq C_+^\alpha(\tpsi)+\left\lceil C_+^\alpha(\tphi_p)\right\rceil_T $, and we conclude by using the continuity and monotonicity of the selectors. 
\end{proof}

\begin{lem}\label{lem:inv par conj}
    Let $C : \widetilde{\conto}(M,\xi)\to\R$ be a continuous map so that for any $\tpsi,\tphi\in\widetilde{\conto}(M,\xi)$
    \begin{enumerate}[1.]
        \item $C(\tphi)\in\spec^\alpha(\tphi)$
        \item $\tphi\cleq\tpsi$ implies that $C(\tphi)\leq C(\tpsi)$. 
    \end{enumerate}
    Then 
    $\left\lceil C(\tpsi\cdot \tphi\cdot\tpsi^{-1})\right\rceil_T=\left\lceil C(\tphi)\right\rceil_T$ and $\left\lfloor C(\tpsi\cdot \tphi\cdot\tpsi^{-1})\right\rfloor_T=\left\lfloor C(\tphi)\right\rfloor_T$.
\end{lem}
The first equality of Lemma \ref{lem:inv par conj} was already proved in \cite[Lemma 2.2]{arlove2024contactnonsqueezingvariousclosed}. The second equality can be proven exactly in the same way. For completeness we present a proof of it below. 
\begin{proof}
    Let $(\psi_t)\subset\conto(M,\xi)$ be a contact isotopy representing $\tpsi$ and denote by $\tpsi_s:=[(\psi_{ts})_{t\in[0,1]}]$ for all $s\in[0,1]$. First, suppose that $\spec^\alpha(\tphi)\cap T\Z=\emptyset$. Since the Reeb flow is $T$-periodic it implies that $\spec^\alpha(\tpsi_s\cdot\tphi\cdot\tpsi_s^{-1})\cap T\Z=\emptyset$ for all $s\in[0,1]$. Therefore, by continuity of the map $s\mapsto C(\psi_s\cdot\tphi\cdot\tpsi_s^{-1})$ we deduce the desired equalities  $\left\lceil C(\tpsi\cdot \tphi\cdot\tpsi^{-1})\right\rceil_T=\left\lceil C(\tphi)\right\rceil_T$ and $\left\lfloor C(\tpsi\cdot \tphi\cdot\tpsi^{-1})\right\rfloor_T=\left\lfloor C(\tphi)\right\rfloor_T$. Let us now turn to the case where $\spec^\alpha(\tphi)\cap T\Z\ne\emptyset$. Since $\spec^\alpha(\tphi)$ is nowhere dense (Lemma 2.11 \cite{albers}), there exists a positive sequence of numbers $(\varepsilon_p)_{p\in\N}$ that goes to $0$ so that $\spec^\alpha(\tphi_p^\pm)\cap T\Z=\emptyset $ where $\tphi_p^+:=\tphi_\alpha^{\varepsilon_p}\cdot \tphi$ and $\tphi_p^-:=\tphi_\alpha^{-\varepsilon_p}\cdot\tphi$. Thus by the first part of the proof we deduce that $\left\lceil C(\tpsi\cdot\tphi_p^-\cdot\tpsi^{-1})\right\rceil_T=\left\lceil C( \tphi_p^-)\right\rceil_T\text{ and } \left\lfloor C(\tpsi\cdot \tphi_p^+\cdot\tpsi^{-1})\right\rfloor_T=\left\lfloor C(\tphi_p^+)\right\rfloor_T$. Moreover, since $\tphi_p^-\cleq\tphi\cleq\tphi_p^+$ and both sequences $(\tphi_p^\pm)_{p\in\N}$ converge to $\tphi$, by monotonicity and continuity of $C$ we deduce the desired equalities.
\end{proof}

\section{Proof of Corollary \ref{cor:1}, Corollary \ref{cor:2} and Corollary \ref{cor:3}}\label{se:preuve des corollaires}

In this Section \ref{se:preuve des corollaires} we prove the corollaries stated in Section \ref{se:geometry} about the geometry of $\widetilde{\conto}(M,\xi)$. Therefore, in Section \ref{se:preuve des corollaires} $(M,\xi)$ is a closed strongly orderable contact manifold and $\alpha$ is a contact form supporting $\xi$ whose Reeb flow is $T$-periodic. The proofs in Section \ref{se:preuve des corollaires} are very similar to the proofs given in Section 4 and Section 5 of \cite{allais2024spectral} (see also \cite{Arlove2023}).
\subsection{Proof of Corollary \ref{cor:1}}\label{se:preuve stably unbounded}
Recall that a conjugation invariant pseudo-norm $\nu :G\to\R_{\geq 0}$ on a group $G$ is by definition a map that satisfies:
\begin{enumerate}[1.]
    \item $\nu(\id)=0$
    \item $\nu(g^{-1})=\nu(g)$
    \item $\nu(g h)\leq \nu(g)+\nu(h)$
    \item $\nu(hgh^{-1})=\nu(g)$.
\end{enumerate}
A conjugation invariant pseudo-norm $\nu : G\to\R$ is said to be stably unbounded if there exists an element $g\in G$ such that the sequence $\left(\frac{\nu(g^k)}{k}\right)_{k\in\N}$ admits a positive limit. 
\begin{proof}[Proof of Corollary \ref{cor:1}]\

Since $C_-^\alpha\leq C_+^\alpha$ it implies that $\mu_{\mathrm{spec}}^\alpha=\max\left\{ \left\lceil C_+^\alpha\right\rceil_T, -\left\lfloor C_-^\alpha\right\rfloor_T\right\}\geq 0$. Consider $\tphi,\tpsi\in\widetilde{\conto}(M,\xi)$ arbitrary elements. 

\begin{enumerate}[1.]
    \item Thanks to Theorem \ref{thm:specselstrongord} $C_\pm^\alpha(\id)=0$ thus $\mu_{\mathrm{spec}}^\alpha(\id)=0$
    \item Thanks to Lemma \ref{lem:poincare} $\left\lceil C_+^\alpha(\tphi)\right\rceil_T =-\left\lfloor C_-^\alpha(\tphi^{-1})\right\rfloor_T$ and thus $\mu_{\mathrm{spec}}^\alpha(\tphi)=\mu_{\mathrm{spec}}^\alpha(\tphi^{-1})$. 
    \item Thanks to Lemma \ref{lem:inegalite triangulaire}
    \[\left\lceil C_+^\alpha(\tphi\cdot\tpsi)\right\rceil_T \leq \left\lceil C_+^\alpha(\tphi)\right\rceil_T + \left\lceil C_+^\alpha(\tpsi)\right\rceil_T\text{ and } \left\lfloor C_-^\alpha(\tphi\cdot\tpsi)\right\rfloor_T\geq \left\lfloor C_-^\alpha(\tphi)\right\rfloor_T +\left\lfloor C_-^\alpha(\tpsi)\right\rfloor_T.\] This implies that $\mu_{\mathrm{spec}}^\alpha(\tphi\cdot\tpsi)\leq \mu_{\mathrm{spec}}^\alpha(\tphi)+\mu_{\mathrm{spec}}^\alpha(\tpsi)$ 
    \item Since $\left\lceil C_\pm^\alpha(\tpsi\cdot \tphi\cdot\tpsi^{-1})\right\rceil_T=\left\lceil C_\pm^\alpha(\tphi)\right\rceil_T$ and $\left\lfloor C_\pm^\alpha(\tpsi\cdot \tphi\cdot\tpsi^{-1})\right\rfloor_T=\left\lfloor C_\pm^\alpha(\tphi)\right\rfloor_T$ by Lemma \ref{lem:inv par conj} we deduce that $\mu_\mathrm{spec}^\alpha(\tpsi\cdot\tphi\cdot\tpsi^{-1})=\mu_\mathrm{spec}^\alpha(\tphi)$. 
\end{enumerate}
Finally, to prove that $\mu_\spec^\alpha$ is stably unbounded, we use the fact $C_\pm^\alpha((\tphi_\alpha^T)^k)=C_\pm^\alpha(\tphi_\alpha^{Tk})=Tk$ for any $k\in \N_{>0}$ thanks to Theorem \ref{thm:specselstrongord}. Hence $\frac{\mu_\mathrm{spec}^\alpha((\tphi_\alpha^T)^k)}{k}=T$. 
\end{proof}

\subsection{Proof of Corollary \ref{cor:2}}\

Recall that $\mathcal{E}$ denotes the set of elements in $\widetilde{\conto}(M,\xi)$ that can be represented by embedded contact isotopies starting at the identity, where we say that a contact isotopy $(\varphi_t)$ is embedded if $0\notin\spec^\alpha(\varphi_t\circ\varphi_s^{-1})$ for all $s\ne t\in[0,1]$. Since we have assumed that the Reeb flow of $\alpha$ is $T$-periodic, if a contact isotopy $(\varphi_t)$ is embedded then $T\Z\cap\spec^\alpha(\varphi_t\circ\varphi_s^{-1})=\emptyset$ for all $s\ne t\in[0,1]$. Therefore we have the following Lemma. 
\begin{lem}\label{lem:discriminant}
    If $\tphi\in\mathcal{E}$, $C_+^\alpha(\tphi)<T$.
\end{lem}
\begin{proof}
    Let $(\varphi_t)$ be a contact isotopy representing $\tphi$ such that $T\Z\cap \spec^\alpha(\varphi_t)=\emptyset$ for all $t\in(0,1]$ (which exists by the above discussion) and denote by $\tphi_s:=[(\varphi_{st})_{t\in[0,1]}]$. Suppose by contradiction that $C_+^\alpha(\tphi)=C_+^\alpha(\tphi_1)>T$. Since $C_+^\alpha(\tphi_0)=C_+^\alpha(\widetilde{\id})=0$ and the map $s\mapsto C_+^\alpha(\tphi_s)$ is continuous, it implies that there exists $s\in[0,1]$ such that $C_+^\alpha(\tphi_s)=T$, which contradicts the fact that $\spec^\alpha(\tphi_s)\cap T\Z=\emptyset$. 
\end{proof}

The triangle inequality property satisfied by $\left\lceil C_+^\alpha\right\rceil$ (see Lemma \ref{lem:inegalite triangulaire}) allows to deduce the following.

\begin{lem}\label{cor:disc}
    For any $\tphi\in\widetilde{\conto}(M,\xi)$ we have $ C_+^\alpha(\tphi) < T\nu_{\mathrm{disc}}(\tphi)$.
\end{lem}

\begin{proof}
    Let $k:=\nu_{\mathrm{disc}}(\tphi)$. By definition it implies that there exists $\tphi_1,\cdots,\tphi_k\in\mathcal{E}$ such that $\tphi=\prod\limits_{i=1}^k\tphi_i$. Thus, $ C_+^\alpha\left(\prod\limits_{i=1}^k\tphi_i\right) \leq C_+^\alpha(\tphi_1)+\sum\limits_{i=2}^k \left\lceil C_+^\alpha(\tphi_i)\right\rceil_T <Tk$, where the first inequality comes from Lemma \ref{lem:inegalite triangulaire} and the last one from Lemma \ref{lem:discriminant}. 
\end{proof}

\begin{proof}[Proof of Corollary \ref{cor:2}]
Since $C_+^\alpha(\tphi_\alpha^{TN})=TN$, for any $N\in\N$, we deduce by Lemma \ref{cor:disc} that $\nu_{\mathrm{disc}}$ is stably unbounded.\\

Let us now show that in the case where all the orbits of the Reeb flow has the same minimal period equals to $T$, then the Reeb flow is a geodesic. Fix $s\in\R_{>0}$. First, let us show that  \begin{equation}\label{eq:longueur discriminante}
\mathcal{L}_\mathrm{disc}((\varphi_\alpha^{ts})_{t\in[0,1]})\leq \left\lfloor \frac{s}{T}\right\rfloor +1.
\end{equation} Let $0<s_0:=\frac{s}{\left\lfloor s/T\right\rfloor +1}<\frac{s}{s/T}=T$. We claim that 
$\mathcal{L}_\mathrm{disc}(\left(\varphi_{\alpha}^{t}
\right)_{t\in[(i-1)s_0,is_0]})=1$ for all $i\in[1,\left\lfloor s/T\right\rfloor+1]\cap\N$. Indeed, if we suppose by contradiction that $\mathcal{L}_\mathrm{disc}(\left(\varphi_{\alpha}^{t}
\right)_{t\in[(i-1)s_0,is_0]})>1$ for some $i\in[1,\left\lfloor s/T\right\rfloor+1]\cap\N$, then it implies that there exists $t_1<t_2\in[(i-1)s_0,is_0] $ such that $0\in\spec^\alpha(\varphi_\alpha^{t_2-t_1})$. This implies the contradiction that the Reeb flow has a periodic orbit of period $0<t_2-t_1<T$ and proves our claim. Now,  the concatenation of the contact isotopies $\left(\varphi_{\alpha}^{t}
\right)_{t\in[(i-1)T_0,iT_0]}$, for $i\in[1,\left\lfloor s/T\right\rfloor+1]\cap\N$, being exactly $(\varphi_\alpha^{ts})_{t\in[0,1]}$ we get the desired inequality \eqref{eq:longueur discriminante}. Moreover, $\nu_\mathrm{disc}(\tphi_\alpha^s)> \frac{C_+^\alpha(\tphi_\alpha^s)}{T}$ by Lemma \ref{cor:disc}. Thus
\begin{equation}\label{eq:norm disc}
\nu_\mathrm{disc}(\tphi_\alpha^s)\geq \left\lfloor \frac{s}{T}\right\rfloor +1. 
\end{equation}
Combining \eqref{eq:longueur discriminante} and \eqref{eq:norm disc} we deduce that the Reeb flow is indeed a geodesic since $\nu_\mathrm{disc}(\tphi_\alpha^s)=\mathcal{L}_\mathrm{disc}((\varphi_\alpha^{ts})_{t\in[0,1]})$.\end{proof}

\subsection{Proof of Corollary \ref{cor:3}}
The proof of Corollary \ref{cor:3} will be very similar to the one of Corollary \ref{cor:2}. We start with the following. 
\begin{lem}\label{cor:osc}
    For any $\tphi\in\widetilde{\conto}(M,\xi)$ we have $C_+^\alpha(\tphi)<T\nu_{\mathrm{osc}}(\tphi)$. 
\end{lem}
\begin{proof}
    Let $k:=\nu^+_{\mathrm{osc}}(\tphi)\leq \nu_\mathrm{osc}(\tphi)$. This means that there exist $N$ monotone and embedded contact isotopies $(\varphi_t^1), \cdots, (\varphi_t^N) $ starting at the identity, with exactly $k$ of them being positive such that $\tphi=\prod\limits_{i=1}^N\tphi_i$, where $\tphi_i:=[(\varphi_t^i)]$. Thus $\tphi\cleq\prod\limits_{i=1}^k\tphi_{\sigma(i)}$ where $\sigma : \N\cap [1,k]\to \N$ is the increasing map so that $\id\cleq \tphi_{\sigma(i)}$. Therefore, 
    \[C_+^\alpha(\tphi)\leq C_+^\alpha\left(\prod\limits_{i=1}^k \tphi_{\sigma(i)}\right)\leq C_+^\alpha(\tphi_{\sigma(1)})+\sum\limits_{i=2}^k\left\lceil C_+^\alpha(\tphi_{\sigma(i)})\right\rceil_T<Tk \]
    where we deduce the first inequality by monotonicity of $C_+^\alpha$, the second one by Lemma \ref{lem:inegalite triangulaire} and the last one by Lemma \ref{lem:discriminant}.
\end{proof}

\begin{proof}
   Since $C_+^\alpha(\tphi_\alpha^{TN})=TN$ for any $N\in\N$ we deduce by Lemma \ref{cor:osc} that $\nu_{\mathrm{osc}}$ is stably unbounded. \\

   Let us now show that in the case when all the orbits of the Reeb flow have the same minimal period equals to $T$, then the Reeb flow is a geodesic. Fix $s\in\R_{>0}$. Since $(\varphi_\alpha^{ts})_{t\in[0,1]}$ is a positive contact isotopy $\mathcal{L}_\mathrm{osc}((\varphi_\alpha^{ts})_{t\in[0,1]})=\mathcal{L}_\mathrm{osc}^+((\varphi_\alpha^{ts})_{t\in[0,1]})$ and one can argue exactly as in the proof of Corollary \ref{cor:2} to show that \begin{equation}\label{eq:longueur osc}
       \mathcal{L}_\mathrm{osc}((\varphi_\alpha^{ts})_{t\in[0,1]})\leq \left\lfloor \frac{s}{T}\right\rfloor +1.
   \end{equation}
   Moreover, by Lemma \ref{cor:osc} we have that $\nu_{\mathrm{osc}}(\tphi_\alpha^s)>\frac{C_+^\alpha(\tphi_\alpha^s)}{T}$ and so 
   \begin{equation}\label{eq:norm osc}
\nu_\mathrm{osc}(\tphi_\alpha^T)\geq \left\lfloor \frac{C_+^\alpha(\tphi_\alpha^s)}{T}\right\rfloor +1=\left\lfloor \frac{s}{T}\right\rfloor +1. 
\end{equation}
Combining \eqref{eq:longueur osc} and \eqref{eq:norm osc} we deduce that the Reeb flow is indeed a geodesic since $\nu_\mathrm{osc}(\tphi_\alpha^s)=\mathcal{L}_\mathrm{osc}((\varphi_\alpha^{ts})_{t\in[0,1]})$.\end{proof}

\section{Proof of the contact big fiber theorems}\label{sec:bft}

To prove Theorem \ref{thm:bft} (resp. Theorem \ref{thm:bft in lens space}), as in \cite{djordjević2025quantitativecontacthamiltoniandynamics,uljarević2025contactquasistatesapplications}, we use the spectral selector of Theorem \ref{thm:specselstrongord} (resp. Proposition \ref{prop:lens space}) to first construct a contact quasi-state and then deduce from it a contact quasi-measure. Let $j\geq 2$ be an integer and $\underline{w}:=(w_0,\cdots,w_n)$ a $n$-tuple of positive integers relatively prime to $j$. In the rest of this Section \ref{sec:bft} $(M,\xi)$ is either:
\begin{enumerate}[1.]
    \item  a closed strongly orderable contact manifold for which we fix an arbitrary contact form $\alpha$ supporting $\xi$
    \item or, the lens space $(L_j^{2n+1}(\underline{w}),\xi_j^{\underline{w}})$ endowed with its standard contact distribution and $\alpha:=\alpha_j^{\underline{w}}$ is the contact form described in Section \ref{sec:bft intro} whose Reeb flow is $T_{\underline{w}}$-periodic.
 \end{enumerate}
 Moreover the map $C_\alpha :\widetilde{\conto}(M,\xi)\to\R$ will denote either the map $c_{\alpha}$ of Proposition \ref{prop:lens space} when $(M,\xi)$ is the lens space or the map $C_+^\alpha$ of Theorem \ref{thm:specselstrongord} otherwise.
 

\subsection{Contact quasi-state}
Recall from the introduction that there exists an isomorphism $F_\alpha : \chi(M,\xi)\to C^\infty(M,\R)$ between contact vector fields and smooth functions given by $X\mapsto \alpha(X)$. To any $h\in C^\infty(M,\R)$ we denote by $\tphi_h$ the element of $\widetilde{\conto}(M,\xi)$ that can be represented by the contact isotopy $(\varphi_t)_{t\in[0,1]}$ starting at the identity and satisfying $\frac{d}{dt}\varphi_t=F_\alpha^{-1}(h)\circ\varphi_t$. We say that $h$ generates $\tphi_h$. Finally, we write $\mathrm{Supp}(h):=\overline{\{x\in M\ |\ h(x)\ne 0\}}$ for its support.

\begin{prop}\label{prop:contact quasi state}
 The map \[\zeta_\alpha : C^\infty(M,\R)\to \R\cup\{\pm\infty\}\quad \quad  h\mapsto\underset{k\to\infty}{\lim\inf} \frac{C_\alpha(\tphi_h^k)}{k} \] 
   satisfies the following for all $g,h\in C^\infty(M,\R)$ and $s\in\R$:
\begin{enumerate}[1.]
    \item (normalization) $\zeta_\alpha(s)=s$
    \item (monotonicity) if $g\leq h$ then $\zeta_\alpha(f)\leq \zeta_\alpha(g)$
      \item (triangle inequality) if $\{h,g\}_\alpha=\{1,h\}_\alpha=\{1,g\}_\alpha=0$ then $\zeta_\alpha(h+g)\leq \zeta_\alpha(h)+\zeta_\alpha(g)$
    \item (vanishing property) if $R_\alpha\cdot h=0$ and there exists $\varphi\in\conto(M,\xi)$ such that $\varphi(\mathrm{Supp}(h))\cap \mathrm{Supp}(h)=\emptyset$ then $\zeta_\alpha(h)=0$.
\end{enumerate}
\end{prop}

\begin{rem}\label{rem:finiteness}
    Note that the first and second property of $\zeta_\alpha$ stated in Proposition \ref{prop:contact quasi-measure} implies that $\zeta_\alpha$ takes finite values. Indeed, for any $h\in C^\infty(M,\R)$, denoting by $m_-:=\min h$, $m_+:=\max h$, we have $m_-=\zeta_\alpha(m_-)\leq\zeta_\alpha(h)\leq \zeta_\alpha(m)=m$.
\end{rem}


Before proving Proposition \ref{prop:contact quasi state} let us state and prove the four following Lemmas. 

\begin{lem}\label{lem:inegalite triangulaire renversée}
    Let $\tphi,\tpsi\in\widetilde{\conto}(M,\xi)$ then 
    \begin{enumerate}[1.]
        \item $C_\alpha(\tphi\cdot \tpsi)\geq C_\alpha(\tphi)-\left\lceil  C_\alpha(\tpsi^{-1})\right\rceil_{T_{\underline{w}}}$ if $(M,\xi)$ is the lens space $(L_j^{2n+1},\xi_j)$
        \item $C_\alpha(\tphi\cdot\tpsi)\geq C_\alpha(\tphi)-\ell_+^{(\Phi\Psi)^*\beta}(\tPsi^{-1}\cdot \tDelta,\tDelta)$ otherwise, where $\beta:=\beta_\alpha$, $\Phi:=\Pi(\widetilde{G_\alpha}(\tphi))$, $\Psi:=\Pi(\widetilde{G_\alpha}(\tpsi))$ and $\tPsi^{-1}:=\widetilde{G_\alpha}(\tpsi^{-1})$ (with the notations  of Section \ref{se:spectral selector intro} and Section \ref{se:def spec sel}). 
    \end{enumerate}
\end{lem}
\begin{proof}
    \begin{enumerate}[1.]
\item Suppose $(M,\xi)$ is the lens space $(L_j^{2n+1}(\underline{w}),\xi_j^{\underline{w}})$. For any $\tphi,\tpsi\in\widetilde{\conto}(M,\xi)$, by Proposition \ref{prop:lens space}, $C_\alpha(\tphi)=C_\alpha((\tphi\cdot\tpsi)\cdot \tpsi^{-1})\leq C_\alpha(\tphi\cdot\tpsi)+ \left\lceil  C_\alpha(\tpsi^{-1})\right\rceil_{T_{\underline{w}}}. $ We thus deduce the desired inequality $C_\alpha(\tphi\cdot\tpsi)\geq C_\alpha(\tphi)- \left\lceil  C_\alpha(\tpsi^{-1})\right\rceil_{T_{\underline{w}}}$.
\item Suppose now that $(M,\xi)$ is a strongly orderable contact manifold which is not a lens space so that $C_\alpha=C_+^\alpha$. For any $\tphi,\tpsi\in\widetilde{\conto}(M,\xi)$, by Theorem \ref{thm:specselstrongord}, $C_+^\alpha(\tphi)=C_+^\alpha((\tphi\cdot\tpsi)\cdot\tpsi^{-1})\leq C_\alpha(\tphi\cdot\tpsi)+ \ell_+^{(\Phi\Psi)^*\beta}(\tPsi^{-1}\cdot\tDelta,\tDelta)$. Thus we get the desired inequality $C_\alpha(\tphi\cdot\tpsi)\geq C_\alpha(\tphi)-\ell_+^{(\Phi\Psi)^*\beta}(\tPsi^{-1}\cdot\tDelta,\tDelta)$. \qedhere
    \end{enumerate}
\end{proof}

Recall that for any $\varphi\in\conto(M,\xi)$ we define its support $\mathrm{Supp}(\varphi)$ to be the closure of $\{x\in M\ |\ \varphi(x)\ne x\}$. Moreover, we say that a contactomorphism $\varphi\in\conto(M,\xi)$ displaces a subset $A\subset M$ if $\varphi(A)\cap A=\emptyset. $

\begin{lem}\label{lem:spec}
    Let $\tphi,\tphi'\in\widetilde{\conto}(M,\xi)$. If $\Pi(\tphi')\in\conto(M,\xi)$ displaces the support of $\Pi(\tphi)$ and $\Pi(\tphi)\circ\varphi_\alpha^t=\varphi_\alpha^t\circ\Pi(\tphi)$ for all $t$ then $\spec^\alpha(\tphi\cdot\tphi')\subset \spec^\alpha(\tphi').$
\end{lem}

The proof of Lemma \ref{lem:spec} repeats almost verbatim the proof of Lemma 5.6 in \cite{arlove2024contactnonsqueezingvariousclosed}.

\begin{proof}[Proof of Lemma \ref{lem:spec}]
Denote by $\varphi:=\Pi(\tphi)$, $\varphi':=\Pi(\tphi')$. Let $t\in\spec^\alpha(\tphi\cdot\tphi')$ and $x\in M$ such that $\varphi(\varphi'(x))=\varphi_\alpha^t(x)$ and $((\varphi\circ\varphi')^*\alpha)_x=\alpha_x$. First, suppose by contradiction that $x\in \mathrm{Supp}(\varphi)$. Then $\varphi'(x)\notin \mathrm{Supp}(\varphi)$ since $\varphi'$ displaces the support of $\varphi$. Therefore, $\varphi(\varphi'(x))=\varphi'(x)=\varphi_\alpha^t(x)$ which implies the contradiction that $\varphi'(x)\in \varphi_\alpha^t(\mathrm{Supp}(\varphi))=\mathrm{Supp}(\varphi_\alpha^{t}\circ\varphi\circ\varphi_\alpha^{-t})=\mathrm{Supp}(\varphi)$ since $\varphi$ commutes with the Reeb flow. Therefore $x\notin\mathrm{Supp}(\varphi)$. And so, arguing as before, $\varphi_\alpha^t(x)\notin\mathrm{Supp}(\varphi)$. Thus $\varphi(\varphi'(x))=\varphi_\alpha^t(x)\notin\mathrm{Supp}(\varphi^{-1})$ and so, composing both sides of the equation by $\varphi^{-1}$, we deduce that $\varphi'(x)=\varphi_\alpha^t(x)$.  Finally, note that $((\varphi\circ\varphi')^*\alpha)_x=((\varphi')^*\alpha)_x=\alpha_x$, since $\varphi'(x)=\varphi_\alpha^t(x)\notin\mathrm{Supp}(\varphi)$, which implies that $t\in\spec^\alpha(\tphi')$ and concludes the proof.\end{proof}



As a consequence of the previous Lemma \ref{lem:spec} we get the following. 

\begin{lem}\label{lem:inegalite capacite-energie}
   Let $\varphi\in\conto(M,\xi)$ and $h\in C^\infty(M,\R)$ be such that $R_\alpha\cdot h=0$ and $\varphi(\mathrm{Supp}(h))\cap\mathrm{Supp}(h)=\emptyset$. Then $C_\alpha(\tphi_h\cdot\tphi)=C_\alpha(\tphi)$ for any $\tphi\in\widetilde{\conto}(M,\xi)$ such that $\Pi(\tphi)=\varphi$.
\end{lem}
\begin{proof}
    Let us denote by $\varphi_h^s\in\conto(M,\xi)$ the map we get by integrating the contact vector field $X_h^\alpha=F_\alpha^{-1}(h)$ until time $s\in\R$. For all $s\in\R$, since $\varphi$ displaces the support of $\varphi_h^s$, we deduce by Lemma \ref{lem:spec} that $\spec^\alpha(\tphi_h^s\cdot\tphi)=\spec^\alpha(\tphi)$, where $\tphi_h^s:=[(\varphi_h^{ts})_{t\in[0,1]}]$. By continuity and spectrality of $C_\alpha$ we deduce that the map $s\mapsto C_\alpha(\tphi_h^s\cdot\tphi)$ has to be constant since it takes value in the nowhere dense set $\spec^\alpha(\tphi)$ (see Lemma 2.11 \cite{albers}). This concludes the proof since $\tphi_h^0\cdot\tphi=\tphi$.
\end{proof}

\begin{lem}\label{lem:subadditive}
    Let $h\in C^\infty(M,\R)$ be such that $R_\alpha\cdot h=0$. Then $\zeta_\alpha(h)=\underset{k\to\infty}\lim \frac{C_\alpha(\tphi_h^k)}{k}$.
\end{lem}
\begin{proof}
     $R_\alpha\cdot h=0$ implies that the contact flow $(\varphi_h^t)_{t\in[0,1]}$ we get by integrating $X_h^\alpha=F_\alpha^{-1}(h)$ is a $\alpha$-strict contact isotopy, i.e. $(\varphi_h^t)^*\alpha=\alpha$. Thus, we deduce that the sequence $\left(\frac{C_\alpha(\tphi_h^k)}{k}\right)_{k\in\N}$ is subadditive and conclude by Feteke Lemma. 
\end{proof}

We are now ready to prove Proposition \ref{prop:contact quasi state}

\begin{proof}[Proof of Proposition \ref{prop:contact quasi state}]\

\begin{enumerate}[1.]
    \item We have $F_\alpha^{-1}(s)=sR_\alpha$ for any $s\in\R$ and thus $\tphi_s=\tphi_\alpha^s$. Hence, $C_\alpha(\tphi_s^k)=C_\alpha(\tphi_\alpha^{sk})=sk$ and this implies the normalization property.
    \item The monotonicity property is a direct consequence of the following implications \[g\leq h\Rightarrow \tphi_g\cleq \tphi_h\Rightarrow \tphi_g^k\cleq \tphi_h^k\Rightarrow C_\alpha(\tphi_g^k)\leq C_\alpha(\tphi_h^k).\]
    The first implication is trivial (see \cite[Lemma 2.2]{FPR}), the second one follows from bi-invariance of $\cleq$ (see \cite{EP00}) and the last one is a consequence of the monotonicity of $C_\alpha$. 
    \item First, consider any $h,g : M\to\R$ and denote by $(\varphi_h^t)_{t\in\R}$ and $(\varphi_g^t)_{t\in\R}$ the contact isotopies we get by integrating the contact vector fields $X_h^\alpha$ and $X_g^\alpha$. Then the smooth function corresponding to the speed $X_s\in\chi(M,\xi)$ of the contact isotopy $(\varphi_h^t\circ\varphi_g^t)_{t\in\R}$ at time $s\in\R$, i.e. $\frac{d}{dt}|_{t=s}(\varphi_h^t\circ\varphi_g^t)=X_s(\varphi_h^s\circ\varphi_g^s)$, is given by
    \[k_s:=F_\alpha^{-1}(X_s)=h+(e^{f_s}\cdot g)\circ \varphi_h^s,\]
    where $f_s$ denotes the $\alpha$-conformal factor of $\varphi_h^s$, i.e. $(\varphi_h^s)^*\alpha=e^{f_s}\alpha$. Suppose now that $\{h,g\}_\alpha=\{1,h\}_\alpha=\{1,g\}_\alpha=0$. From $\{h,g\}_\alpha=\{1,h\}_\alpha=0$ we get that $k_s=h+g$. Therefore $\tphi_h^k\cdot\tphi_g^k=\tphi_{h+g}^k$ for any $k\in\N_{>0}$. Moreover $\{1,h\}_\alpha=0$ implies that $\tphi_h^k$ can be represented by the $\alpha$-strict contact isotopy $(\varphi_h^{kt})$ and so 
    \begin{equation}\label{eq:inegalité triangulaire quasi-state}
C_\alpha(\tphi_{h+g}^k)=C_\alpha(\tphi_{h}^k\cdot\tphi_g^k)\leq  C_\alpha(\tphi_h^k)+C_\alpha(\tphi_g^k).
    \end{equation}
    Since $\{1,h\}_\alpha=\{1,g\}_\alpha=0$ we also have that implies that $R_\alpha\cdot(h+g)=0$. Thus, by Lemma \ref{lem:subadditive}, we obtain the desired inequality
    \[\zeta_\alpha(h+g)\leq \zeta_\alpha(h)+\zeta_\alpha(g)\]it by dividing the inequality \eqref{eq:inegalité triangulaire quasi-state} by $k$ and by taking the limit as $k\to+\infty$. 
    \item Let $h\in C^\infty(M,\R)$, $k\in\N$ and $\varphi\in\conto(M,\xi)$ such that $R_\alpha\cdot h=0$ and $\varphi(\mathrm{Supp}(h))\cap\mathrm{Supp}(h)=\emptyset$. A direct computation shows that $kh$ generates the element $\tphi_h^k\in\widetilde{\conto}(M,\xi)$. In particular, by Lemma \ref{lem:inegalite capacite-energie}
    \begin{equation}\label{eq:égalité capacité énergie}
C_\alpha(\tphi_h^k\cdot\tphi)=C_\alpha(\tphi)
\end{equation}  for any $\tphi\in\widetilde{\conto}(M,\xi)$ such that $\Pi(\tphi)=\varphi$. Let us write $\tphi_h^k$ as $(\tphi_h^k\cdot\tphi)\cdot\tphi^{-1}$. If $(M,\xi)$ is the lens space $(L_j^{2n+1}(\underline{w}),\xi_j^{\underline{w}})$, using the above equality \eqref{eq:égalité capacité énergie}, we thus get
\begin{equation}\label{eq:lens space}
\begin{aligned} C_\alpha(\tphi_h^k)&\leq C_\alpha(\tphi_h^k\cdot\tphi)+ \left\lceil C_\alpha(\tphi^{-1})\right\rceil_{T_{\underline{w}}}= C_\alpha(\tphi) + \left\lceil C_\alpha(\tphi^{-1})\right\rceil_{T_{\underline{w}}},\ \text{ and} \\
C_\alpha(\tphi_h^k)&\geq  C_\alpha(\tphi_h^k\cdot\tphi)- \left\lceil  C_\alpha(\tphi)\right\rceil_{T_{\underline{w}}}= C_\alpha(\tphi) - \left\lceil C_\alpha(\tphi)\right\rceil_{T_{\underline{w}}},
\end{aligned}
\end{equation}
where the inequality of the first line follows from Proposition \ref{prop:lens space} and the inequality of the second line from Lemma \ref{lem:inegalite triangulaire renversée}.
Otherwise, if $(M,\xi)$ is strongly orderable but is not a lens space, $C_\alpha(\tphi_h^k)=C_+^\alpha(\tphi_h^k)$ and we write $\tPhi_h^k:=\widetilde{G_\alpha}(\tphi_h^k),\ \Phi_h^k=\Pi(\tPhi_h^k),\ \tPhi:=\widetilde{G_\alpha}(\tphi),\ \Phi:=\Pi(\tPhi)$.  Similarly, we get by Theorem \ref{thm:specselstrongord}, Lemma \ref{lem:inegalite triangulaire renversée} and \eqref{eq:égalité capacité énergie}
\begin{equation}\label{eq:stronglyorderable}
\begin{aligned}
C_+^\alpha(\tphi_h^k)&\leq C_+^\alpha(\tphi_h^k\cdot\tphi)+\ell_+^{(\Phi_h^k\circ\Phi)^*\beta}(\tPhi^{-1}\cdot \tDelta,\tDelta)=C_+^\alpha(\tphi)+\ell_+^{\Phi^*\beta}(\tPhi^{-1}\cdot \tDelta,\tDelta)\\
C_+^\alpha(\tphi_h^k)&\geq C_+^\alpha(\tphi_h^k\cdot\tphi)-\ell_+^{(\Phi_h^k\circ\Phi)^*\beta}(\tPhi^{-1}\cdot \tDelta,\tDelta)=C_+^\alpha(\tphi)-\ell_+^{\Phi^*\beta}(\tPhi^{-1}\cdot \tDelta,\tDelta)
\end{aligned}
\end{equation}
  where the last equalities follows from the fact that $(\Phi_h^k\circ\Phi)^*\beta=\Phi^*((\Phi_h^k)^*\beta)$ and $(\Phi_h^k)^*\beta=\beta$ since $R_\alpha\cdot h=0$ (see for example the sixth point in the proof of Theorem \ref{thm:specselstrongord} in Section \ref{se:preuves des thm}). Dividing the inequalities in \eqref{eq:lens space} and \eqref{eq:stronglyorderable} by $k$, we deduce that $\frac{C_\alpha(\tphi_h^k)}{k}$ goes to $0$ as $k$ goes to infinity. \qedhere
    \end{enumerate}
\end{proof}

\subsection{Contact quasi-measure}

Let $\zeta_\alpha : C^\infty(M,\R)\to\R$ be the map defined in \ref{prop:contact quasi state} (see also Remark \ref{rem:finiteness}). Denoting by $\mathfrak{P}(M)$ the power set of $M$ we have the following.

\begin{prop}\label{prop:contact quasi-measure}
    The map 
    \[\tau_\alpha :\mathfrak{P}(M)\to \R\quad \quad A\mapsto \sup\left\{\zeta_\alpha(h)\ \left|\  \parbox{5cm}{$h\in C^\infty(M,[0,1])$ such that  $R_\alpha\cdot h=0$ and  $\mathrm{Supp}(h)\subset A$}\right.\right\}\]
    satisfies:
    \begin{enumerate}[1.]
        \item (normalization) $\tau_\alpha(M)=1$
        \item (monotonicity) for any $A,B\in\mathfrak{P}(M)$, if $A\subset B$ then $\tau_\alpha(A)\leq \tau_\alpha(B)$
        \item (subadditivity) for any $N,k\in\N_{>0}$, if $f: M\to \R^N$ is a $\alpha$-contact involutive map and $X_1,\cdots, X_k$ are bounded open subsets of $\R^N$ 
        \[\tau_\alpha\left(\underset{i\in[1,k]\cap \N}\bigcup f^{-1}(X_i)\right)\leq \sum\limits_{i=1}^k\tau_\alpha(f^{-1}(X_i))\]
        \item (vanishing property) if $A\subset M$ can be displaced by some $\varphi\in\conto(M,\xi)$, i.e. $\varphi(A)\cap A=\emptyset$, $\tau_\alpha(A)=0$.
    \end{enumerate}
\end{prop}

   Before proving Proposition \ref{prop:contact quasi-measure} let us fix an arbitrary $N\in\N_{>0}$ and prove the following Lemmas.
    \begin{lem}\label{lem:contact quasi-measure}
        Let $f: M\to \R^{N}$ be a smooth map such that $\ud f(R_\alpha)\equiv 0$. Then for any open subset $X\subset \R^N$
        \[\tau_\alpha(f^{-1}(X))=\sup\{\zeta_\alpha(\mu\circ f)\ |\ \mu\in C^{\infty}(\R^N,[0,1]),\ \mathrm{Supp}(\mu)\subset X\}.\]
    \end{lem}
The above Lemma \ref{lem:contact quasi-measure} is inspired by \cite[Lemma 8.1]{djordjević2025quantitativecontacthamiltoniandynamics}.

\begin{proof}[Proof of Lemma \ref{lem:contact quasi-measure}]
It is clear from the definition of $\tau_\alpha$ that $\sup\{\zeta(\mu\circ f)\ |\ \mu\in C^{\infty}(\R^N,[0,1]),\ \mathrm{Supp}(\mu)\subset X\}\leq \tau_\alpha(f^{-1}(X))$. To show the reverse inequality, let $\varepsilon>0$ be a positive number and let $h\in C^\infty(M,[0,1])$ be a function supported in $f^{-1}(X)$ such that $R_\alpha\cdot h=0$ with \begin{equation}\label{eq:lem contact quasi-measure}
\tau_\alpha(f^{-1}(X))-\varepsilon\leq \zeta_\alpha(h).
\end{equation} Since $M$ is compact and $\mathrm{Supp}(h)$ is closed by definition, the set $f(\mathrm{Supp}(h))$ is a compact set of the open set $X$. Hence, one can construct a smooth function $\mu :\R^N\to [0,1]$ which is supported in $X$ and equal to $1$ on $f(\mathrm{Supp}(h))$. Thus $h\leq \mu\circ f$ and so $
\zeta_\alpha(h)\leq \zeta_\alpha(\mu\circ f)$ by monotonicity of $\zeta_\alpha$. Combining the previous inequality with \eqref{eq:lem contact quasi-measure}, and letting $\varepsilon$ goes to $0$, allows to get the desired reverse inequality.  
\end{proof}

\begin{lem}\label{lem:poisson bracket}
    Let $f=(f_1,\cdots,f_N): M\to\R^N$ be a $\alpha$-contact involutive map and $\mu,\mu': \R^N\to\R$ be smooth functions. Then $\{\mu'\circ f,\mu\circ f\}_\alpha=\{1,\mu\circ f\}_\alpha=\{1,\mu'\circ f\}_\alpha=0$. 
\end{lem}

\begin{proof}
First, let us show that $\{1,\mu\circ f\}_\alpha=0$. Indeed, since $\{1,f_i\}_\alpha=\ud f_i(R_\alpha)=0$ for all $i\in[1,N]\cap \N$ we have
\[\{1,\mu\circ f\}_\alpha=\ud (\mu\circ f)(R_\alpha)=\ud_f\mu\circ\ud f(R_\alpha)=\ud _f\mu\circ(\ud f_1(R_\alpha),\cdots,\ud f_N(R_\alpha))=0.\]
The same arguments obviously holds to show that $\{1,\mu'\circ f\}_\alpha=0$. \\

Let us show now that $\{\mu'\circ f,\mu\circ f\}_\alpha=0$. The dimension of $M$ being $2n+1$, by Darboux Theorem, for any point $p$, there exists a neighborhood $U\subset M$ of $p$ and a diffeomorphism $\varphi=(x_1,\cdots,x_n,y_1,\cdots,y_n,z) : U\to\R^{2n+1}$ such that $\alpha|_U=\ud z-\sum\limits_{i=1}^ny_i\ud x_i$. A direct computation shows that $R_\alpha|_{U}=\frac{\partial}{\partial z}$ and 
    \begin{equation}\label{eq:contact vf}
X_h^\alpha|_{U}=\left(\sum\limits_{j=1}^n\left(\frac{\partial h}{\partial x_j}+y_j\frac{\partial h}{\partial z}\right)\frac{\partial}{\partial y_j}-\frac{\partial h}{\partial y_j}\frac{\partial}{\partial x_j}\right)+\left(h-\sum\limits_{j=1}^ny_j\frac{\partial h}{\partial y_j}\right)\frac{\partial}{\partial z},
    \end{equation}
    for any smooth function $h: M\to \R$.
    Therefore, if $h :M\to\R$ satisfies $\{1,h\}_\alpha=\ud h(R_\alpha)=0$ it implies in particular that $\{1,h\}_\alpha|_U=\frac{\partial h}{\partial z}=0$ and the above expression \eqref{eq:contact vf} simplifies to
\begin{equation*}\label{eq:contact vf RI}
X_h^\alpha|_{U}=\left(\sum\limits_{j=1}^n\frac{\partial h}{\partial x_j}\frac{\partial}{\partial y_j}-\frac{\partial h}{\partial y_j}\frac{\partial}{\partial x_j}\right)+\left(h-\sum\limits_{j=1}^ny_j\frac{\partial h}{\partial y_j}\right)\frac{\partial}{\partial z}.
    \end{equation*}
    Thus if $h, h':M\to\R$ are two smooth functions satisfying $\{1,h\}_\alpha=\{1,h'\}_\alpha=0$, a direct computation yields 
    \begin{equation*}
        \{h',h\}_\alpha|_U=\ud h(X_{h'}^\alpha)|_U=\sum\limits_{j=1}^n\left(-\frac{\partial h}{\partial x_j}\frac{\partial h'}{\partial y_j}+\frac{\partial h}{\partial y_j}\frac{\partial h'}{\partial x_j}\right).
    \end{equation*}
    Denoting by $(a_1,\cdots,a_N)$ the coordinate functions on $\R^N$ and applying the previous equality  to $h=\mu\circ f$, $h'=\mu'\circ f$  leads to 
    \[\begin{aligned}
        \{\mu'\circ f,\mu\circ f\}_\alpha|_U&=\sum\limits_{j=1}^n\left(-\frac{\partial (\mu\circ f)}{\partial x_j}\frac{\partial (\mu'\circ f)}{\partial y_j}+\frac{\partial (\mu\circ f)}{\partial y_j}\frac{\partial (\mu'\circ f)}{\partial x_j}\right)\\
    &=\sum\limits_{j=1}^n\left(-\left(\sum\limits_{k=1}^N\frac{\partial\mu}{\partial a_k}\frac{\partial f_k}{\partial x_j}\right)\left(\sum\limits_{l=1}^N\frac{\partial \mu'}{\partial a_l}\frac{\partial f_l}{\partial y_j}\right)+\left(\sum\limits_{l=1}^N\frac{\partial \mu'}{\partial a_l}\frac{\partial f_l}{\partial x_j}\right)\left(\sum\limits_{k=1}^N \frac{\partial \mu}{\partial a_k}\frac{\partial f_k}{\partial y_j}\right)\right)\\
&=\sum\limits_{j=1}^n\left(\sum\limits_{k=1}^N\sum\limits_{l=1}^N \frac{\partial \mu}{\partial a_k}\frac{\partial \mu'}{\partial a_l}\left(-\frac{\partial f_k}{\partial x_j}\frac{\partial f_l}{\partial y_j}+\frac{\partial f_k}{\partial y_j}\frac{\partial f_l}{\partial x_j}\right)\right)\\
&=\sum\limits_{k=1}^N\sum\limits_{l=1}^N \frac{\partial \mu}{\partial a_k}\frac{\partial \mu'}{\partial a_l}\left(\sum\limits_{j=1}^n\left(-\frac{\partial f_k}{\partial x_j}\frac{\partial f_l}{\partial y_j}+\frac{\partial f_k}{\partial y_j}\frac{\partial f_l}{\partial x_j}\right)\right)\\
&=\sum\limits_{k=1}^N\sum\limits_{l=1}^N \frac{\partial \mu}{\partial a_k}\frac{\partial \mu'}{\partial a_l}\{f_l,f_k\}_\alpha|_U=0.\qedhere
    \end{aligned}\]

\end{proof}

The proof of Proposition \ref{prop:contact quasi-measure} we give below repeats almost verbatim the proof of Lemma 8.2 and Lemma 8.3 of \cite{djordjević2025quantitativecontacthamiltoniandynamics}.

\begin{proof}[Proof of Proposition \ref{prop:contact quasi-measure}]
The normalization property of $\tau_\alpha$ follows directly from the normalization and monotonicity of $\zeta_\alpha$. The monotonicity of $\tau_\alpha$ is a  direct consequence of its definition. The vanishing property of $\tau_\alpha$ follows also immediatly from the vanishing property of $\zeta_\alpha$. Therefore it remains to prove the subadditivity property.\\

Note that, by applying an induction argument, it is enough to prove that \[\tau\left(f^{-1}(X_1)\cup f^{-1}(X_2))\leq \tau(f^{-1}(X_1))+\tau(f^{-1}(X_2)\right)\] for any two bounded open subsets $X_1, X_2\subset \R^N$. Let $\varepsilon>0$ be a positive number. By Lemma \ref{lem:contact quasi-measure} there exists $\mu:\R^n\to[0,1]$ a smooth function supported in $U:=X_1\cup X_2$ so that \begin{equation}\label{eq:preuve subadditivity 1}
\tau(f^{-1}(X_1)\cup f^{-1}(X_2))=\tau(f^{-1}(X_1\cup X_2))\leq\zeta_\alpha(\mu\circ f)+\varepsilon.
\end{equation}
Consider now two smooth functions $\mu_1,\mu_2:\R^N\to [0,1]$ so that $\mu_1$ is supported in $X_1$, $\mu_2$ is supported in $X_2$ and $\mu=\mu_1+\mu_2$. To construct such functions $\mu_1,\mu_2$, one may take a partition of unity on $U$ subordinate to the open cover $\{X_1,X_2\}$. More precisely, let $\rho_1, \rho_2 : U \to[0,1]$ be two smooth functions supported in $X_1$ and $X_2$ respectively, so that $\rho_1+\rho_2=1$. Thus, for $i\in\{1,2\}$, one can define \[\mu_i: \R^N\to [0,1] \quad \quad  x\mapsto \begin{cases}
\rho_i(x)\mu(x) & \text{ if } x\in U\\
0 & \text{ otherwise}.
\end{cases}
\]
Therefore we have
\begin{equation}\label{eq:preuve subadditivity 2} \zeta_\alpha(\mu\circ f)= \zeta_\alpha(\mu_1\circ f+\mu_2\circ f)\leq \zeta_\alpha(\mu_1\circ f)+\zeta_\alpha(\mu_2\circ f)
\end{equation}
where the inequality is a consequence of the triangle inequality property that we can apply thanks to Lemma \ref{lem:poisson bracket}. Combining the relations \eqref{eq:preuve subadditivity 1} and \eqref{eq:preuve subadditivity 2} we get 
\[\begin{aligned}\tau(f^{-1}(X_1)\cup f^{-1}(X_2))\leq \zeta_\alpha(\mu\circ f)+\varepsilon&\leq \zeta_\alpha(\mu_1\circ f)+\zeta_\alpha(\mu_2\circ f)+\varepsilon\\
&\leq \tau(f^{-1}(X_1))+\tau(f^{-1}(X_2))+\varepsilon
\end{aligned}\]
where the last inequality follows again from Lemma \ref{lem:contact quasi-measure}.
Since $\varepsilon>0$ can be chosen arbitrarily small, this concludes the proof. \end{proof}


We are now ready to prove Theorem \ref{thm:bft} and Theorem \ref{thm:bft in lens space}.
\begin{proof}[Proof of Theorem \ref{thm:bft} and Theorem \ref{thm:bft in lens space}]\

\textbf{First step:} We claim that if the fiber $A:=f^{-1}\{y\}$ over some $y\in \R^N$ is displaceable by $\varphi\in\conto(M,\xi)$ then there exists a bounded open neighborhood $U$ of $y$ such that $f^{-1}(U)$ is also displaceable by $\varphi$. Indeed, since $A\cap\varphi(A)=\emptyset$, there exists open neighborhoods $V$ and $V'$, of $A$ and $\varphi(A)$ respectively, so that $V\cap V'=\emptyset$. Moreover, by continuity of $\varphi$, for any point $a\in A$ there exists a neighborhood $V_a\subset V$ such that $\varphi(V_a)\subset V'$. Finally, since $M$ is compact there exists a bounded open neighborhood $U$ of $y$ such that $f^{-1}(U)\subset W:=\underset{a\in A}\bigcup V_a$, and thus
\[\varphi(f^{-1}(U))\cap f^{-1}(U)\subset \varphi(W)\cap W\subset V'\cap V=\emptyset\]
which proves our claim.

\textbf{Second step:} Suppose by contradiction that the fiber $f^{-1}\{y\}$ is displaceable for all $y\in\R^N$. Then by the first step, for all $y\in\R^N$, there exists a bounded open neighborhood $U_y\subset \R^N$ of $y$ so that $f^{-1}(U_y)$ is displaceable. Since $M$ is compact, there exists an integer $k\in\N_{>0}$ and $y_1,\cdots,y_k\in\R^N$ so that $\underset{i\in[1,k]\cap \N}\bigcup f^{-1}(U_{y_i})=M$ and thus leads to the following contradiction 
\[1=\tau_\alpha(M)=\tau\left(\underset{i\in[1,k]\cap \N}\bigcup f^{-1}(U_{y_k})\right)\leq \sum\limits_{i=1}^k\tau\left(f^{-1}(U_{y_k})\right)\]
since $\tau_\alpha(f^{-1}(U_{y_k}))=0$ by the vanishing property of $\tau_\alpha$.  
\end{proof}

\section{Proof of Proposition \ref{prop:weinstein}}\label{se:weinstein preuve}
In this Section we prove Proposition \ref{prop:weinstein}. Let us recall that a contact form $\beta$ in a contact manifold $(G,\Xi)$ is complete if its Reeb vector field is complete.

\begin{lem}\cite[Lemma 3.2]{allais2023spectral}\label{lem:invariance signe}
    Let $\Lambda$ be a closed Legendrian of a cooriented contact manifold $(G,\Xi)$ such that $\uLeg(\Lambda)$ is orderable and $\beta$ a complete contact form supporting $\Xi$. For any smooth function $f : N\to\R$, if $\gamma:=e^f\beta$ is a complete contact form
\begin{equation*}
        \begin{cases}
            e^{\inf f}\ell_\pm^\beta(\Lambda_1,\Lambda_0)
            \leq \ell_\pm^{\gamma}(\Lambda_1,\Lambda_0) \leq
            e^{\sup f}\ell_\pm^\beta(\Lambda_1,\Lambda_0)
            &\text{when } \ell_\pm^{\gamma}(\Lambda_1,\Lambda_0)\leq 0,\\
            e^{\sup f}\ell_\pm^\beta(\Lambda_1,\Lambda_0)
            \leq \ell_\pm^\gamma(\Lambda_1,\Lambda_0) \leq
            e^{\inf f}\ell_\pm^\beta(\Lambda_1,\Lambda_0)
            &\text{when }\ell_\pm^\gamma(\Lambda_1,\Lambda_0)\leq 0.
        \end{cases}
    \end{equation*}
\end{lem}

\begin{proof}[Proof of Proposition \ref{prop:weinstein}]
First note that since $C_-^\alpha\leq C_+^\alpha$, the map $N_{\mathrm{spec}}^\alpha$ takes value in $\R_{\geq 0}$.  \\

Let $\tphi\in\widetilde{\conto}(M,\xi)$ so that $N_{\mathrm{spec}}^\alpha(\tphi)=0$. Then it implies that $C_+^\alpha(\tphi)=C_-^\alpha(\tphi)=0$ and so $\Pi(\tphi)=\id$. Hence, $N_{\mathrm{spec}}^\alpha(\tpsi)>0$ implies that $\tpsi\in\widetilde{\conto}(M,\xi)\setminus\pi_1(\conto(M,\xi))$.\\

Finally, let us consider $\tphi\in\pi_1(\conto(M,\xi))$ so that $c:=N_{\mathrm{spec}}^\alpha(\tphi)>0$. Since $\Pi(\tphi)=\id$ it implies, by spectrality of $C_\pm^\alpha$, that there exists $x\in M$ such that  $\varphi_\alpha^c(x)=x$, i.e. the $\alpha$-Reeb flow admits a non trivial periodic orbit. Moreover, for any contact form $\alpha'$ supporting $\xi$ and its coorientation, there exists a smooth function $f: M\to\R$ such that $\alpha'=e^f\alpha$ and so by Lemma \ref{lem:invariance signe} and \eqref{eq:def du selecteur spectral} we deduce that $c':=N_{\mathrm{spec}}^{\alpha'}(\tphi)>0$. Hence, arguing as before it implies that there exists $x'\in M$ such that $\varphi_{\alpha'}^{c'}(x')=x'$ and so the $\alpha'$-Reeb flow admits a non trivial periodic orbit. 
\end{proof}

\bibliographystyle{amsplain}
\bibliography{biblio}

\end{document}